\documentclass[a4paper,fleqn]{cas-sc}

\usepackage[numbers]{natbib}

\def\tsc#1{\csdef{#1}{\textsc{\lowercase{#1}}\xspace}}
\tsc{WGM}
\tsc{QE}

\usepackage{subcaption}
\usepackage{enumitem}

\newcommand{\R}{\mathbb R}
\newcommand{\Sph}{\mathbb S}
\newcommand{\Kc}{\mathcal K}
\newcommand{\Pc}{\mathcal P}
\newcommand{\Gc}{\mathbb G}
\newcommand{\eps}{\varepsilon}
\newcommand{\de}{\mathsf d}
\newcommand{\1}{\mathbf 1}
\newcommand{\cl}{\operatorname{cl}}
\newcommand{\Int}{\operatorname{Int}}
\newcommand{\conv}{\operatorname{conv}}
\newcommand{\supp}{\operatorname{supp}}
\newcommand{\argmax}{\operatorname*{arg\,max}}
\newcommand{\HD}{\operatorname{HD}}
\newcommand{\Med}{\operatorname{Med}}
\newcommand{\dist}{\operatorname{dist}}

\numberwithin{equation}{section}

\newtheorem{theorem}{Theorem}[section]
\newtheorem{lemma}[theorem]{Lemma}
\newtheorem{proposition}[theorem]{Proposition}
\newtheorem{corollary}[theorem]{Corollary}

\newtheorem{definition}[theorem]{Definition}
\newtheorem{example}[theorem]{Example}
\newtheorem{remark}[theorem]{Remark}

\newproof{pf}{Proof}

\begin{document}

\shorttitle{Depth based on translates of a convex body}    

\shortauthors{Molchanov and Nagy}  

\title [mode = title]{Depth based on translates of a convex body:\\geometry, floating bodies, and empirical theory}  



%

\author[1]{Ilya Molchanov}[orcid=0000-0003-2502-0334]

\cormark[1]


\ead{ilya.molchanov@unibe.ch}



\affiliation[1]{organization={University of Bern, Institute of Mathematical Statistics and Actuarial Science},
            city={Bern},
            postcode={3012}, 
            country={Switzerland}}

\author[2]{Stanislav Nagy}[orcid=0000-0002-8610-4227]

\fnmark[1]

\ead{nagy@karlin.mff.cuni.cz}



\affiliation[2]{organization={Charles University,
	Faculty of Mathematics and Physics},
            city={Prague},
            postcode={186~75}, 
            country={Czech Republic}}

\cortext[1]{Corresponding author}

\fntext[1]{The work of S. Nagy was supported by the ERC CZ grant LL2407 of the Ministry of Education, Youth and Sport of the Czech Republic.}

\nonumnote{This work was initiated during the trimester program ``The Interplay between High-Dimensional Geometry and Probability'' at the Hausdorff Research Institute for Mathematics in Bonn. The authors acknowledge the contribution of Petra Laketa to a preliminary version of this manuscript, and are grateful to Elisabeth Werner for discussions at early stages of this work.}

\begin{abstract}
  For a fixed full-dimensional convex body $K\subset\R^d$, we study
  the $K$-depth of a point $y$ with respect to a probability measure
  $\mu$ by comparing the masses of translates of $K$ that fail to
  contain $y$.  Its upper level sets are intersections of high-mass
  translates of $K$, and hence are convex and admit an exact
  representation as Minkowski erosions. We distinguish the closed core
  determined by the maximisers of the translation-mass function from
  the effective support on which the depth exceeds its unavoidable
  baseline. We establish continuity, affine equivariance, comparison
  with halfspace (Tukey) depth, as well as normal-cone and smoothness criteria. 
  For uniform measures on convex bodies, the regions coincide exactly
  with the ball floating bodies of Sch\"utt, Werner and Yalikun; for
  positive continuous densities they coincide with the weighted ball
  floating bodies of Werner and Yalikun. Their volume-derivative
  results yield relative affine and $L_p$-relative surface-area
  asymptotics for the present depth regions. For dilates $tK$, the
  depth converges to a depth generated by the tangent cones of $K$ as
  $t \to \infty$; for smooth $K$ this limit is Tukey's halfspace
  depth, whereas nonsmooth testing bodies can produce a genuinely
  different limit. The worst-case maximal $K$-depth is exactly
  $1/h_T(K)$, where $h_T(K)$ is the Helly number of the family of
  translates of $K$; for empirical measures on $n$ observations the
  exact constant is $\lceil n/h_T(K)\rceil/n$. Under a uniform
  interior rolling-ball condition and compact support and having zero
  mass on hyperplanes, the convergence to halfspace depth is monotone
  and uniform with an explicit error bound. If the density is bounded
  by $L$, then this bound is of order $O(t^{-1})$
  and this order is sharp. We also record that the dependence on the
  testing body is not monotone in general, while Hausdorff
  perturbations of $K$ give stable depth functions and regions at
  regular levels. Finally, we prove uniform consistency and pointwise
  weak limits for empirical $K$-depth, including the non-Gaussian
  limit caused by nonunique active translates.
\end{abstract}



\begin{keywords}
ball convexity \sep depth \sep floating body \sep halfspace depth \sep $K$-hulls
\end{keywords}

\maketitle

\section{Introduction}

A central region associated with a probability measure is often
constructed as an intersection of members of a prescribed family of
sets having large probability content.  Tukey's halfspace depth uses
closed halfspaces and produces convex depth regions
\cite{Tukey1975,DonohoGasko1992,RousseeuwRuts1999}. The relation
between halfspace-depth regions and classical convex floating bodies
established a link between multivariate statistics and convex
geometry; it was developed systematically by Nagy, Sch\"utt and Werner
\cite{NagySchuettWerner2019}. In this paper the testing family
consists instead of all translates of one fixed convex body $K$.

This construction relates to $K$-convexity. The $K$-hull of a set
$A\subset\R^d$ is the intersection of all translates of $K$ containing
$A$. For a Euclidean ball this is the classical ball hull
\cite{Bezdek2013,FodorKeveiVigh2014}. General $K$-hulls are studied in
\cite{Polovinkin1996} and their random counterparts in
\cite{MarynychMolchanov2022}, among other sources. The present depth
regions replace the requirement that a translate contain the entire
set by a probability-content requirement with respect to a given
probability measure $\mu$ on $\R^d$.

Recent work of Sch\"utt, Werner and Yalikun introduced $C$-ball
floating bodies and, for Euclidean balls, the relative affine surface
area \cite{SchuettWernerYalikun2025}. Werner and Yalikun subsequently
introduced weighted ball floating bodies and $L_p$-relative surface
areas, the latter appearing in the limits of measure deficits, see
\cite{WernerYalikun2026}. The present construction contains these
bodies as exact special cases for uniform or continuously weighted
distributions $\mu$ on $C$-ball-convex supports. It also extends the
construction to arbitrary Borel probability measures $\mu$ and
arbitrary convex testing bodies $K$. In the reverse direction, the
geometric volume-derivative theorems from those works give sharp
small-level asymptotics for our depth regions. Very recent work of
Oliva, Werner and Yalikun \cite{OlivaWernerYalikun2026} develops the
exterior, illumination-body counterpart of this ball-convex
programme.

The main contributions are as follows.
\begin{enumerate}[label=(\roman*),leftmargin=2.2em]
\item The depth regions are identified with Minkowski erosions of $K$
  by high-mass shift sets. This yields convexity, upper
  semicontinuity, a convex-hull reduction, and a precise description
  of the unavoidable positive baseline of the raw depth.
\item The exact distribution-free smallest value if the depth is
  $1/h_T(K)$, where $h_T(K)$ is the translative Helly number of $K$.
  The corresponding exact finite-sample constant is
  $\lceil n/h_T(K)\rceil/n$.
\item Under the natural boundary-null condition on $\mu$, the depth
  $\de_K(y;\mu)$ is continuous in $y$. More strongly, if $\mu_n$
  converges weakly to $\mu$, then the depths converge uniformly.
\item The boundary of a nondegenerate depth region is described
  through its active translates. A normal-cone formula gives a sharp
  smoothness criterion.
\item Under dilation of the testing body, the limiting testing sets
  are the tangent cones of $K$. Smooth testing bodies recover
  halfspace depth $\HD$ in the limit, namely, $\de_{tK}$ decreases to
  halfspace depth $\HD$ uniformly with rate $O(t^{-1})$ as
  $t \to \infty$, and the resulting order is sharp.
\item The map taking $K$ to its depth region $D_\alpha^K(\mu)$ is not
  monotone in general. We also prove stability of the depth function
  and of regular depth regions under Hausdorff perturbations of the
  testing body.
\item The empirical depth is uniformly consistent. At a fixed point,
  its asymptotic law is Gaussian when the active translate is unique
  and is the negative supremum of a Gaussian process over the active
  set in the nonunique case.
\end{enumerate}


\section{Definitions and basic properties}

\subsection{Notation and \texorpdfstring{$K$}{K}-hulls}
\label{sec:notat-texorpdfstr-hu}

Let $\Kc^d$ denote the family of all nonempty compact convex subsets
of $\R^d$, including lower-dimensional sets. We call its elements
compact convex sets. Let $\Kc^d_\circ\subset\Kc^d$ denote those
members that have nonempty interior; only these sets are called convex
bodies. Unless explicitly stated otherwise, $K$ denotes the fixed
testing body and belongs to $\Kc^d_\circ$ throughout the paper.  We
write $\Pc(\R^d)$ for the space of Borel probability measures on
$\R^d$ and denote by $|K|$ the Lebesgue measure of $K$.  For $m\ge1$
and $r>0$, let $B_r^m:=\{z\in\R^m:\|z\|_2\le r\}$ be the
$m$-dimensional Euclidean ball of radius $r$. In the ambient space
$\R^d$ we abbreviate $B_r^d$ to $B_r$ and write $B=B_1^d$.  We write
$\kappa_m:=|B_1^m|$ for $m\ge1$ and use the convention $\kappa_0=1$.
For sets $A,B\subset\R^d$, write $A+B:=\{a+b:a\in A,\ b\in B\}$ and
$\check B=-B:=\{-b:b\in B\}$. We use the following reflected version of
the Minkowski difference:
\begin{equation}
 K\ominus B
 :=\{y:\check B+y\subset K\}
 =\bigcap_{x\in B}(K+x).
 \label{eq:minkowski-difference}
\end{equation}
The intersection, and hence $K\ominus B$, is set equal to $\R^d$ when
$B=\varnothing$. A set $Q\in\Kc^d$ is a \emph{Minkowski summand} of
$K$ if $K=Q+M$ for some $M\in\Kc^d$. Then $K\ominus\check{M}=Q$. 

For $K\in\Kc^d_\circ$ and $A\subset\R^d$, the $K$-hull of $A$ is
$\operatorname{conv}_K(A) :=\bigcap_{x:\,A\subset K+x}(K+x)$, with the
usual empty-index convention. A set is $K$-convex if it coincides with
its $K$-hull; see Figure~\ref{fig:khull}.

\begin{figure}[ht]
\centering
\begin{subfigure}{0.45\textwidth}
 \centering\includegraphics[width=\linewidth]{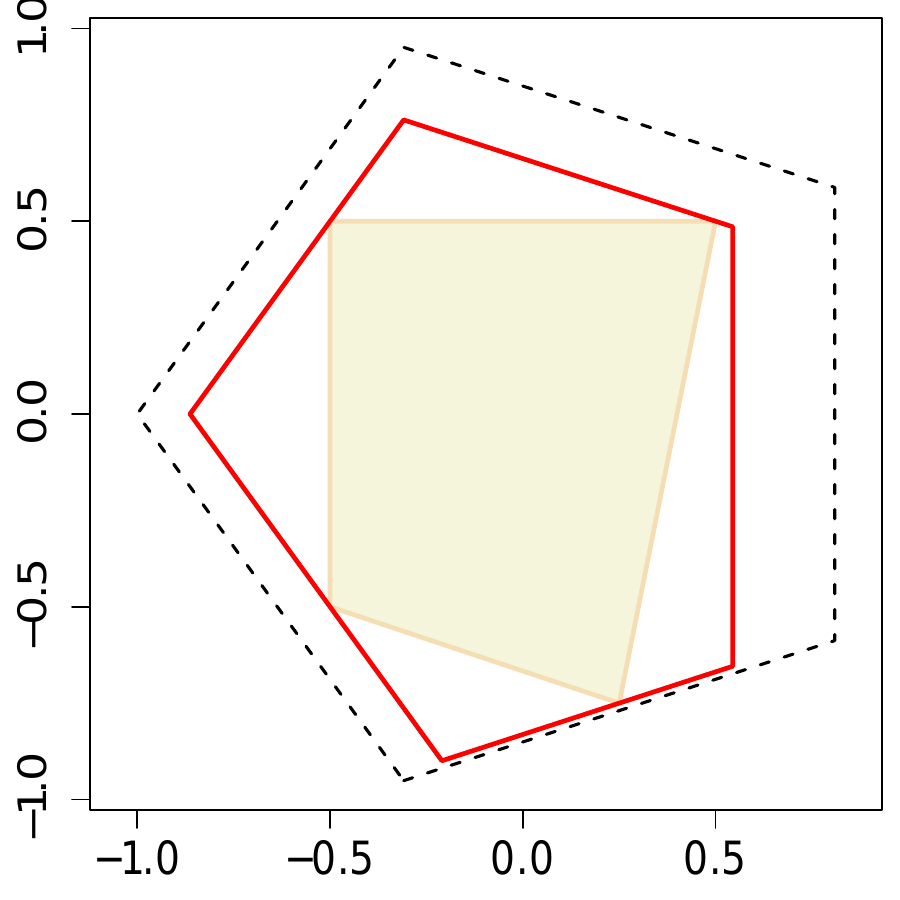}
 \caption{$K$ a regular pentagon.}
\end{subfigure}\hfill
\begin{subfigure}{0.45\textwidth}
 \centering\includegraphics[width=\linewidth]{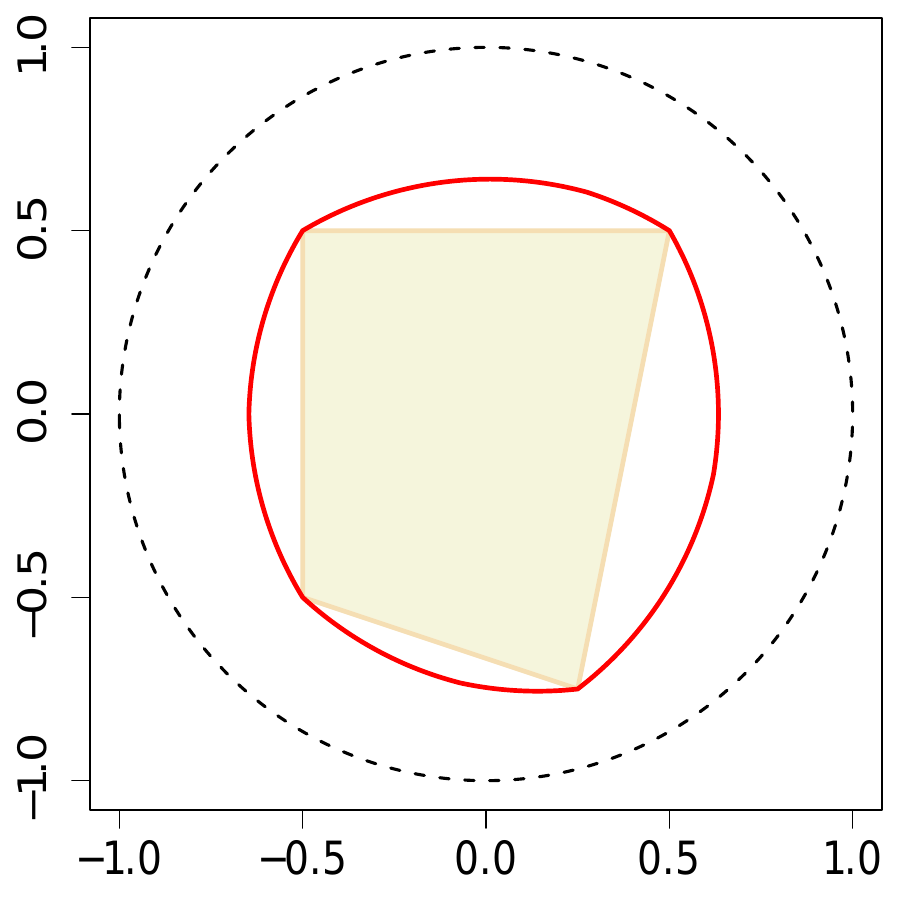}
 \caption{$K$ a Euclidean ball.}
\end{subfigure}
\caption{Section~\ref{sec:notat-texorpdfstr-hu}: The $K$-hull (thick red boundary) of a quadrilateral (beige)
for two testing bodies $K$ (black dashed boundary).}
\label{fig:khull}
\end{figure}

\subsection{Translation mass, depth, and depth regions}

Let $\mu$ be a Borel probability measure on $\R^d$, and let $K$ be the
fixed testing body specified above. Its \emph{translation-mass
  function} is $f_{K,\mu}(x):=\mu(K+x)$ for $x\in\R^d$. When $K$ and
$\mu$ are clear, we write simply $f$.

\begin{lemma}\label{lem:f-basic}
  The function $f$ is upper semicontinuous and satisfies $f(x)\to0$ as
  $\|x\|\to\infty$. Consequently,
  \begin{displaymath}
    \beta_{K,\mu}:=\sup_{x\in\R^d}f(x)
  \end{displaymath}
  is positive and attained, and the maximiser set
  \begin{displaymath}
    A_{K,\mu}:=\{x:f(x)=\beta_{K,\mu}\}
  \end{displaymath}
  is nonempty and compact.
\end{lemma}
\begin{pf}
  If $x_n\to x$, then for every $\eps>0$ and all sufficiently large
  $n$, $K+x_n\subset (K+x)+\eps B$, where $B$ is the closed Euclidean
  unit ball. Hence,
  \begin{displaymath}
    \limsup_n f(x_n)
    =\limsup_{n\to\infty}\mu(K+x_n)
    \le \mu((K+x)+\eps B).
  \end{displaymath}
  The right0-hand side converges to $f(x)$ as $\eps\downarrow0$.

  Choose a compact set $C$ with $\mu(C)>1-\eps$. Then $(K+x)\cap C$ is
  empty when $\|x\|$ is sufficiently large, so that
  $f(x)\le\mu(C^c)<\eps$.  Upper semicontinuity and decay at infinity
  imply attainment and compactness of the maximiser set. Positivity
  follows by choosing a translate whose interior meets $\supp\mu$.
\end{pf}

For $\beta\in[0,1]$, define the strict and nonstrict upper level sets
$L^{>}_\beta:=\{x:f(x)>\beta\}$ and
$L^{\ge}_\beta:=\{x:f(x)\ge\beta\}$. For $\alpha\in(0,1]$, the
  \emph{$K$-depth region} is
\begin{equation*}
  D_\alpha^K(\mu):=\bigcap_{x\in L^{>}_{1-\alpha}}(K+x).
\end{equation*}
The \emph{$K$-depth} of $y\in\R^d$ is
\begin{equation}
  \de_K(y;\mu):=1-M_K(y;\mu),
  \qquad
  M_K(y;\mu):=\sup_{x:\,y\notin K+x}\mu(K+x).
  \label{eq:depth}
\end{equation}
Thus $M_K(y;\mu)$ is the largest mass of a translate that misses $y$.

Small \cite[Definition~2.1]{Small1987} associates with a probability
measure $\mu$ and a prescribed class $\mathcal F$ of measurable sets the
$\mathcal F$-index
 $I_{\mathcal F}(y;\mu)
 :=\sup\{\mu(F):F\in\mathcal F,\ y\notin F\}$. For the translation class
$\mathcal F_K=\{K+x:x\in\R^d\}$, this index equals $M_K(y;\mu)$.
Consequently, the depth in \eqref{eq:depth} is precisely the
complementary normalisation
$\de_K(y;\mu)=1-I_{\mathcal F_K}(y;\mu)$. The complement is used so
that larger values indicate greater centrality and upper level sets are
central regions.

\begin{proposition}
  \label{prop:erosion}
  For every $\alpha\in(0,1]$,
  \begin{equation}
    D_\alpha^K(\mu)
    =K\ominus L^{>}_{1-\alpha}
    =K\ominus\cl\conv L^{>}_{1-\alpha}
    =\{y:\de_K(y;\mu)\ge\alpha\}.
    \label{eq:erosion-level}
  \end{equation}
  Consequently, $\de_K(\cdot;\mu)$ is upper semicontinuous and
  quasi-concave, and the regions $D_\alpha^K(\mu)$ are closed, convex,
  and decreasing in $\alpha$.
\end{proposition}
\begin{pf}
  The first identity is \eqref{eq:minkowski-difference}. Since $K$ is
  closed and convex, it contains $\check L+y$ if and only if it
  contains $\cl\conv(\check L+y)$, proving the convex-hull
  reduction. Finally, $\de_K(y;\mu)<\alpha$ holds precisely when some
  translate missing $y$ has mass greater than $1-\alpha$, which is
  equivalent to $y\notin D_\alpha^K(\mu)$.
\end{pf}

\begin{proposition}
  \label{prop:baseline}
  If $f_{K,\mu}(x)>1-\alpha$ for some $x$, then
  $D_\alpha^K(\mu)$ is compact. We have
  \begin{equation}
    1-\beta_{K,\mu}\le \de_K(y;\mu)\le1,
    \qquad y\in\R^d.
    \label{eq:baseline}
  \end{equation}
  If $x_0\in A_{K,\mu}$ and $y\notin K+x_0$, then
  $\de_K(y;\mu)=1-\beta_{K,\mu}$. In particular,
  \begin{equation}
    D_\alpha^K(\mu)=\R^d
    \quad\text{for every }\alpha\le1-\beta_{K,\mu}.
    \label{eq:trivial-levels}
  \end{equation}
\end{proposition}
\begin{pf}
  The first condition implies that $D_\alpha^K(\mu)$ is nonempty and
  so is compact being intersection of compact sets. 
  The upper bound in~\eqref{eq:baseline} is immediate. Since every
  admissible translate has mass at most $\beta_{K,\mu}$,
  \eqref{eq:baseline} follows. If a maximising translate misses $y$,
  then the defining supremum equals $\beta_{K,\mu}$. Formula
  \eqref{eq:trivial-levels} follows from the strict superlevel
  convention.  
\end{pf}

The  $K$-depth differs from most classical depths because it need not
vanish at infinity. 
The \emph{maximal $K$-depth} is denoted by 
\begin{displaymath}
  \alpha_K^*(\mu):=\sup_{y\in\R^d}\de_K(y;\mu).
\end{displaymath}

\begin{proposition}
  If $\alpha_K^*(\mu)>1-\beta_{K,\mu}$, then the \emph{$K$-median
    region}
  \[
    \Med_K(\mu):=\{y:\de_K(y;\mu)=\alpha_K^*(\mu)\}.
  \]
  is a nonempty compact convex set. If
  $\alpha_K^*(\mu)=1-\beta_{K,\mu}$, the depth is identically equal to
  its baseline.
\end{proposition}
\begin{pf}
  A maximising sequence with depth bounded away from the baseline is
  contained in a translate of $K$, which is compact. Upper
  semicontinuity yields a maximiser, and the upper level set at the
  maximal value is closed and convex. If the supremum equals the lower
  bound in \eqref{eq:baseline}, equality holds everywhere.
\end{pf}

\subsection{Invariance and comparison with halfspace depth}

The halfspace depth is
\begin{displaymath}
  \HD(y;\mu):=\inf_{H\ni y}\mu(H),
\end{displaymath}
where $H$ ranges over closed halfspaces. Its nonstrict upper
$\alpha$-level set is denoted by $D_\alpha^{\mathrm{HD}}(\mu)$.

\begin{proposition}[Basic equivariances and Tukey comparison]
  \label{prop:equivariance}
  Let $a,b\in\R^d$, let $A:\R^d\to\R^d$ be nonsingular, and let $\nu$
  be the image of $\mu$ under $z\mapsto Az+b$. Then
  \begin{align}
    \de_{K+a}(y;\mu)
    &=\de_K(y;\mu),
      \label{eq:K-translation}\\
    \de_{AK}(Ay+b;\nu)
    &=\de_K(y;\mu),
      \label{eq:affine-depth}\\
    D_\alpha^{AK}(\nu)
    &=A D_\alpha^K(\mu)+b.
      \label{eq:affine-region}
  \end{align}
  Moreover, $D_\alpha^{\mathrm{HD}}(\mu)\subset D_\alpha^K(\mu)$, and
  $\HD(y;\mu)\le\de_K(y;\mu)$.
\end{proposition}
\begin{pf}
  Translating $K$ merely reparametrises its translates, proving
  \eqref{eq:K-translation}. The affine image of $K+x$ is $AK+Ax+b$,
  which proves \eqref{eq:affine-depth} and \eqref{eq:affine-region}.

  Let $y\in D_\alpha^{\mathrm{HD}}(\mu)$ and suppose that $y\notin K+x$ for a
  translate with $\mu(K+x)>1-\alpha$. Strict separation gives a closed
  halfspace $H$ with $K+x\subset H$ and $y\notin H$. There exists a
  closed halfspace $G$ containing $y$ and disjoint from
  $H$. Therefore, $\mu(G)\le1-\mu(H)<\alpha$, contradicting
  $\HD(y;\mu)\ge\alpha$. Thus $y\in D_\alpha^K(\mu)$.
\end{pf}

\subsection{The sharp centerpoint constant}

The comparison with halfspace depth yields the universal lower bound
$\alpha_K^*(\mu)\geq 1/(d+1)$, which follows from the result available
for the half-space depth and going back to \cite{MR21962}.  The truly
sharp lower bound is determined by the intersection combinatorics of
the translates of the testing body. This argument is the
Helly-number version of the centerpoint principle for Small's
set-family index; see also \cite{bas:oer17} for various Helly-type
centerpoint results.

\begin{definition}
  The \emph{translative Helly number} $h_T(K)$ is the smallest integer
  $h$ such that every finite family of translates of $K$ with empty
  intersection contains a subfamily of at most $h$ members whose
  intersection is empty.
\end{definition}

Ordinary Helly's theorem (see \cite[Theorem~1.1.6]{Schneider2014})
gives
\begin{equation} 
  2\le h_T(K)\le d+1.
\label{eq:translative-helly-range}
\end{equation}
The lower bound follows because every single translate is nonempty and
two sufficiently distant translates are disjoint.

\begin{theorem}
  \label{thm:sharp-centerpoint}
  For every convex body $K$, the infimum of $\alpha_K^*(\mu)$ over all
  probability measure $\mu\in\Pc(\R^d)$ is attained and is equal to
  $1/h_T(K)$. Furthermore, 
  \begin{equation}
    D_{1/h_T(K)}^K(\mu)\ne\varnothing.
    \label{eq:centerpoint-lower}
  \end{equation}
\end{theorem}
\begin{pf}
  Put $h=h_T(K)$ and consider the family
  $\mathcal G:=\{K+x:\mu(K+x)>1-1/h\}$. If intersection
  $D_{1/h}^K(\mu)$ is empty, then compactness of $K+x$ first yields a
  finite subfamily with empty intersection, and the definition of $h$
  then yields $G_1,\ldots,G_m$ in $\mathcal G$, with $m\le h$, whose
  intersection is empty. Their complements cover $\R^d$, whereas
  $\mu(G_i^c)<1/h$. Consequently,
  \begin{displaymath}
    1 =\mu\left(\bigcup_{i=1}^mG_i^c\right) \le\sum_{i=1}^m\mu(G_i^c)
    <\frac{m}{h} \le1,
  \end{displaymath}
  a contradiction. This proves \eqref{eq:centerpoint-lower}.
  Minimality of $h=h_T(K)$ yields translates $G_i=K+x_i$, with
  $i=1,\ldots,h$, such that their total intersection is empty while
  every proper subfamily has nonempty intersection. Choose $p_i$ from
  the intersection of $G_j$, $j\neq i$.  The point $p_i$ does not
  belong to $G_i$, since otherwise it would belong to all $h$
  translates. For $\mu_0=\frac1h\sum_{i=1}^h\delta_{p_i}$, each $G_i$
  contains precisely the selected points other than $p_i$, so
  $\mu_0(G_i)=1-1/h$. Given $y\in\R^d$, at least one $G_i$ misses $y$.
  Then $\de_K(y;\mu_0)\le1-\mu_0(G_i)=1/h$. Together with the lower
  bound already proved, this shows that $\alpha_K^*(\mu_0)=1/h$.
\end{pf}

\begin{corollary}[Exact finite-sample constant]
  \label{cor:empirical-centerpoint}
  Let $\mu_n=\frac1n\sum_{j=1}^n\delta_{X_j}$ be the empirical measure
  of $X_1, \dots, X_n \in \R^d$, where repeated observations are
  allowed. Then
  \begin{equation}
    \inf_{X_1,\ldots,X_n\in\R^d}\alpha_K^*(\mu_n)
    =\frac{\lceil n/h_T(K)\rceil}{n}.
    \label{eq:empirical-centerpoint}
  \end{equation}
\end{corollary}
\begin{pf}
  Write $h=h_T(K)$. Every mass $\mu_n(K+x)$ is an integer multiple of
  $1/n$. Hence every value of the empirical depth, and also its
  maximum, belongs to the same finite
  grid. Theorem~\ref{thm:sharp-centerpoint} gives
  $\alpha_K^*(\mu_n)\ge1/h$, and therefore
  $\alpha_K^*(\mu_n)\ge \lceil n/h\rceil/n$. For sharpness, use the
  points $p_1,\ldots,p_h$ from the proof of
  Theorem~\ref{thm:sharp-centerpoint}. Choose nonnegative integers
  $n_1,\ldots,n_h$ with sum $n$ such that every $n_i$ is either
  $\lfloor n/h\rfloor$ or $\lceil n/h\rceil$, and place $n_i$
  observations at $p_i$. Then $\mu_n(G_i)=1-n_i/n$.  For each $y$, one
  of the sets $G_i$ misses $y$, and hence
  $\de_K(y;\mu_n)\le n_i/n \le \lceil n/h\rceil/n$. This proves
  equality in \eqref{eq:empirical-centerpoint}.
\end{pf}

\begin{example}
  If $K$ is a parallelotope, then $h_T(K)=2$. Indeed, after an affine
  change of coordinates, its translates are axis-parallel boxes. A
  pairwise intersecting finite family of such boxes has intersecting
  coordinate projections, and the one-dimensional Helly property for
  intervals gives a common point. 
  Proposition~\ref{prop:rolling-helly} below shows, in contrast, that
  a rolling-ball condition for smooth bodies $K$ already forces
  $h_T(K)=d+1$.
\end{example}

\subsection{Continuity}

The measure $\mu$ is \emph{$K$-smooth} if $\mu(\partial(K+x))=0$ for
all $x\in\R^d$. Every measure that is absolutely continuous with
respect to Lebesgue measure is $K$-smooth.

\begin{proposition}
  \label{prop:depth-continuity}
  If $\mu$ is $K$-smooth, then $f_{K,\mu}$ and $y\mapsto\de_K(y;\mu)$
  are continuous.  Consequently, every point
  $y\in\partial D_\alpha^K(\mu)$ satisfies $\de_K(y;\mu)=\alpha$.
\end{proposition}
\begin{pf}
  For $x_n\to x$, the indicators $\1_{K+x_n}$ converge pointwise to
  $\1_{K+x}$ outside $\partial(K+x)$. Dominated convergence confirms
  the continuity of $f_{K,\mu}$.
  
  Put $\overline M(y):=\max\{f(x):y\notin\Int K+x\}$. The maximum is
  attained because the feasible set is closed and $f$ vanishes at
  infinity. We show that $\overline M$ is continuous.  Let $y_n\to
  y$. For an upper bound, take maximisers $x_n$. If $f(x_n)\to0$,
  there is nothing to prove. Assume that $f(x_n)$ has a subsequence
  which is bounded away from zero and pass to this subsequence. Then
  the sequence $(x_n)$ is bounded. Closedness of
  $\{(y,x):y\notin\Int K+x\}$ gives $y\notin\Int K+x$, and continuity
  of $f$ yields $\limsup_n\overline M(y_n)\le\overline M(y)$. For a
  lower bound, fix $x$ such that $y\notin \Int K+x$ and set
  $x_n=x+y_n-y$. Then $y_n-x_n=y-x$, so $x_n$ is feasible for $y_n$,
  and $f(x_n)\to f(x)$. Taking the supremum over feasible $x$ gives
  lower semicontinuity. Hence $\de_K=1-\overline M$ is continuous.
  
  The supremum in \eqref{eq:depth} equals $\overline M(y)$. Indeed, a
  translate $\Int K+x$ can be moved arbitrarily slightly so that it
  misses $y$, and the mass changes continuously.  The boundary claim
  follows from~\eqref{eq:erosion-level}.
\end{pf}

The next lemma supplies a uniform form of the usual continuity-set
principle. For a set $A$, write $A^r:=A+rB$ and
$A^{-r}:=\{z:z+rB\subset A\}$.

\begin{lemma}
  \label{lem:uniform-translates}
  If $\mu_n$ converges weakly to $\mu$ and $\mu$ is $K$-smooth, then
  \begin{displaymath}
    \sup_{x\in\R^d}
    \bigl|\mu_n(K+x)-\mu(K+x)\bigr|\to 0.
  \end{displaymath}
\end{lemma}
\begin{pf}
  We first claim that
  $b_K(r) :=\sup_{x\in\R^d}\mu\bigl((\partial K+x)^r\bigr) \to 0$ as
  $r\downarrow0$. If not, there are $r_j\downarrow0$ and $x_j$ for
  which the displayed masses stay bounded away from zero. Tightness
  excludes $\|x_j\|\to\infty$. Along a subsequence, $x_j\to x$, and
  \[
    (\partial K+x_j)^{r_j} \subset(\partial K+x)^{r_j+\|x_j-x\|}.
  \]
  Continuity from above and $K$-smoothness contradict the assumed
  lower bound.

  Let $\rho_n$ be the Prokhorov distance between $\mu_n$ and $\mu$, so
  $\rho_n\to0$. The defining inequalities for the Prokhorov metric
  give, for every $x$, that
  $\mu_n(K+x) \le \mu((K+x)^{\rho_n})+\rho_n$, and
  $\mu_n(K+x) \ge \mu((K+x)^{-\rho_n})-\rho_n$. The outer and inner
  boundary layers in these inequalities are contained in
  $(\partial K+x)^{\rho_n}$. Therefore,
  $\sup_x|\mu_n(K+x)-\mu(K+x)| \le \rho_n+b_K(\rho_n)\to 0$.
\end{pf}

\begin{theorem}
  \label{thm:joint-continuity}
  If $\mu_n$ converges weakly to $\mu$, $\mu$ is $K$-smooth, and
  $y_n\to y$, then 
  \begin{equation}
    \sup_{z\in\R^d}
    |\de_K(z;\mu_n)-\de_K(z;\mu)|
    \le \sup_{x\in\R^d}|\mu_n(K+x)-\mu(K+x)| \to 0.
    \label{eq:depth-lipschitz-measure}
\end{equation}
\end{theorem}
\begin{pf}
  For fixed $z$, the two suprema defining $M_K(z;\mu_n)$ and
  $M_K(z;\mu)$ are taken over the same set of shifts. The difference
  of the suprema is bounded by the right-hand side of
  \eqref{eq:depth-lipschitz-measure}. Apply
  Lemma~\ref{lem:uniform-translates} and then
  Proposition~\ref{prop:depth-continuity}.
\end{pf}

\section{Effective support and structured distributions}

\subsection{Effective support and \texorpdfstring{$K$}{K}-smoothness}

Define the closed maximiser core and the effective support by
\begin{equation}
  C_K(\mu):=\bigcap_{x\in A_{K,\mu}}(K+x)=K\ominus A_{K,\mu},
  \qquad
  S_K(\mu):=\{y:\de_K(y;\mu)>1-\beta_{K,\mu}\}.
  \label{eq:cores}
\end{equation}
Then $S_K(\mu)$ is convex, being the increasing union of the nested
convex sets $D_\alpha^K(\mu)$. Furthermore, \eqref{eq:erosion-level}
implies that
\begin{displaymath}
  S_K(\mu)=\bigcup_{\alpha>1-\beta_{K,\mu}}D_\alpha^K(\mu) \subset
  C_K(\mu).
\end{displaymath}
  The closed core
$C_K(\mu)$ and the effective support $S_K(\mu)$ in \eqref{eq:cores}
need not coincide. The difference is caused by maximising translates
that are active only in the limit.

\begin{proposition}
  \label{prop:effective-support}
  Let $\beta=\beta_{K,\mu}$ and $A=A_{K,\mu}$. For every probability
  measure we have $S_K(\mu)\subset C_K(\mu)$. If $\mu$ is
  $K$-smooth, then also $S_K(\mu)=\bigcap_{x\in A}\Int K+x$; in
  particular, $S_K(\mu)$ is open and convex. If the maximiser of the
  translation-mass function is unique, say $A=\{x_0\}$, then
  $S_K(\mu)=\Int K+x_0$.  
\end{proposition}
\begin{pf}
  If $y$ is not in one of the maximising translates, its depth equals
  the baseline; hence $S_K(\mu)\subset C_K(\mu)$.  If $\mu$ is
  $K$-smooth, then $f$ is continuous. By \eqref{eq:depth}, a point $y$
  belongs to $S_K(\mu)$ precisely when
  \begin{equation}
    \sup_{x\notin y+\check K}f(x)<\beta.
    \label{eq:support-sup}
  \end{equation}
  If \eqref{eq:support-sup} holds and $x\in A$, then
  $x\in\Int(y+\check K)$. Otherwise, $x$ can be approached by points
  outside $y+\check K$, and continuity of $f$ would make the supremum
  in \eqref{eq:support-sup} equal to $\beta$.

  Conversely, suppose that $A\subset\Int\check K+y$ and that the
  supremum in \eqref{eq:support-sup} equals $\beta$. There is a
  sequence $x_n\notin y+\check K$ with $f(x_n)\to\beta$. Since $f$
  vanishes at infinity, the sequence $(x_n)$ is bounded. Any
  convergent subsequence has a limit $x$ outside $\Int\check K+y$,
  while continuity gives $f(x)=\beta$.  This contradicts
  $A\subset\Int \check K+y$.

  Finally, the set $\bigcap_{x\in A}\Int K+x$ is open.  Indeed, if
  $A\subset\Int\check K+y$, compactness of $A$ gives a positive
  distance from $A$ to the complement of $\Int\check K+y$, and the
  same inclusion persists for all points close to $y$.
\end{pf}

\begin{example}[An empty effective support]
  Let $K$ be a Euclidean ball of radius $r$, and let $\mu$ assign mass
  $1/2$ uniformly to each of two balls of radius $r$ whose centres are
  at distance greater than $4r$. A translate of $K$ meets at most
  one component and can capture an entire component. Hence
  $\beta=1/2$.  For every $y$, one of the two maximising translates
  misses $y$, so $\de_K(y;\mu)=\frac12$ for all $y\in\R^d$,
  and $S_K(\mu)=\varnothing$.
\end{example}

\begin{example}[One-dimensional discontinuous depth]
  \label{ex:one-dimens-disc}
  Let $K=[-1,1]$. Consider the mixture of a symmetric triangular law
  on $[1,3]$ with weight $0.3$, an atom of mass $0.2$ at $5$, and a
  shifted standard exponential law on $[6,\infty)$ with weight $0.5$;
  see Figure~\ref{fig:univariate}. Its translation-mass function is
  upper semicontinuous but discontinuous, attains its maximum at the
  shift $x_0=6$, and has $\beta=0.2+0.5(1-e^{-1})\approx0.516$.  The
  depth therefore equals the baseline $1-\beta$ outside
  $K+x_0=[5,7]$. At the endpoint $5$, the depth has an upward jump,
  and the effective support is $[5,7)$ rather than the closed core
  $[5,7]$.  This illustrates why the two notions in \eqref{eq:cores}
  must be kept separate.

  \begin{figure}[ht]
    \centering
    \begin{subfigure}{0.48\textwidth}
      \centering\includegraphics[width=\linewidth]{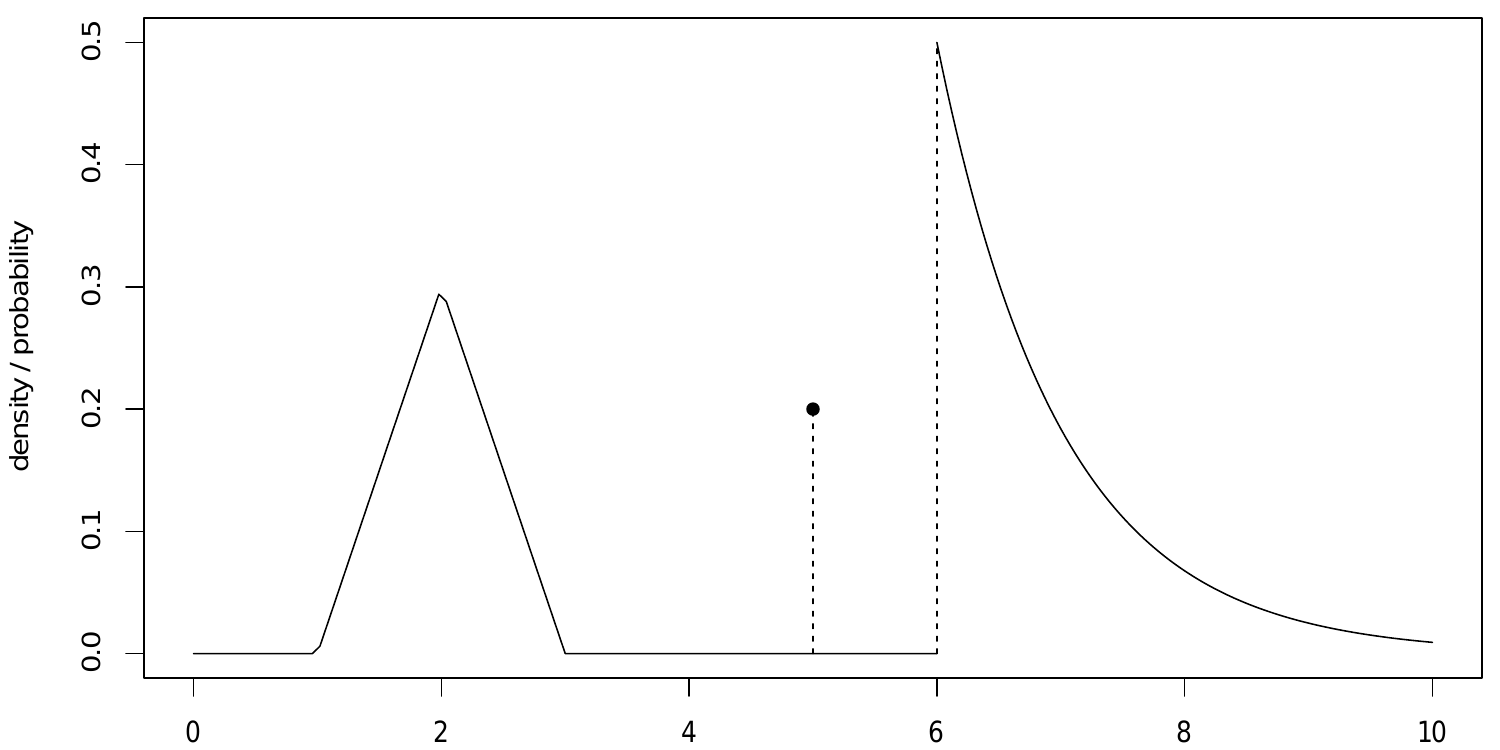}
      \caption{Continuous density and the atom at $5$.}
    \end{subfigure}\hfill
    \begin{subfigure}{0.48\textwidth}
      \centering\includegraphics[width=\linewidth]{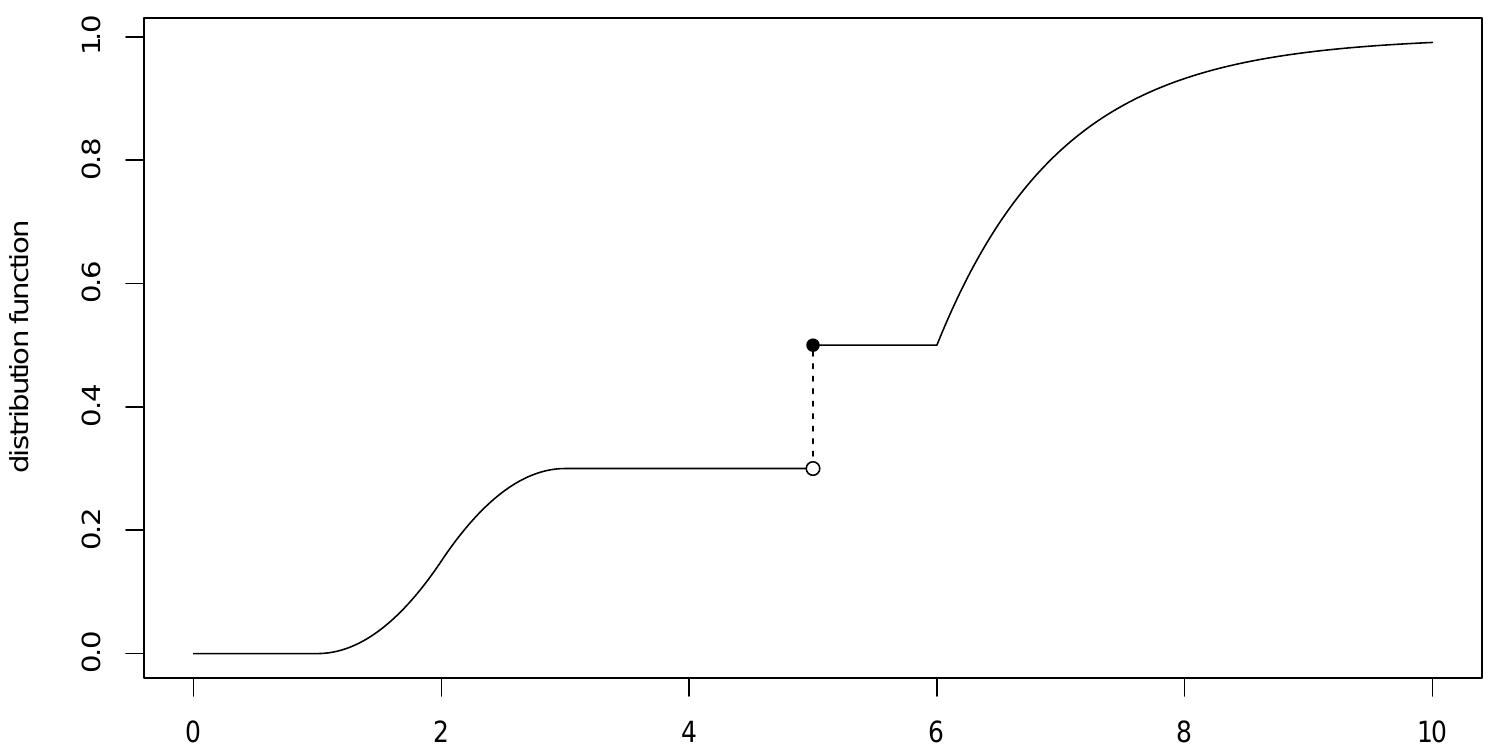}
      \caption{Distribution function.}
    \end{subfigure}

    \medskip
    \begin{subfigure}{0.48\textwidth}
      \centering\includegraphics[width=\linewidth]{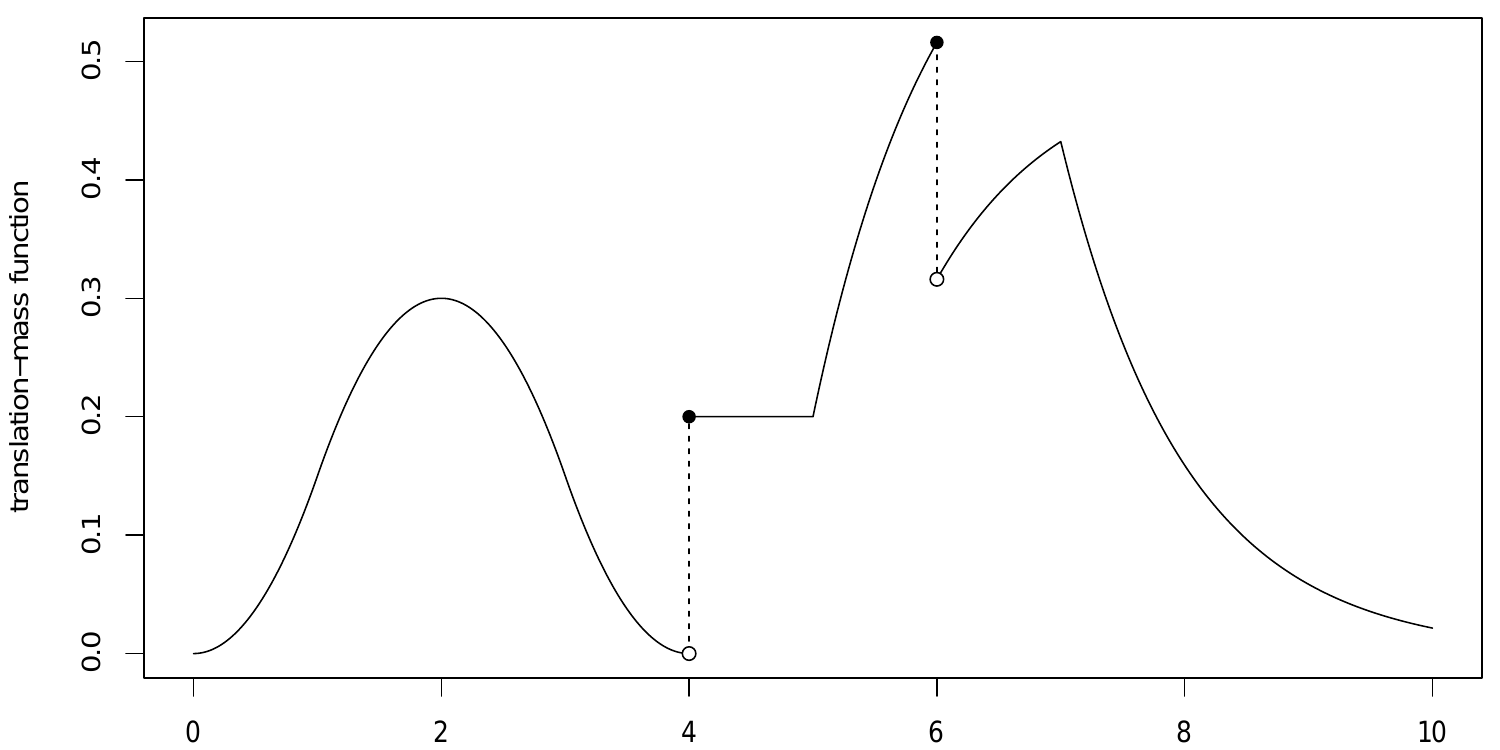}
      \caption{Translation-mass function.}
    \end{subfigure}\hfill
    \begin{subfigure}{0.48\textwidth}
      \centering\includegraphics[width=\linewidth]{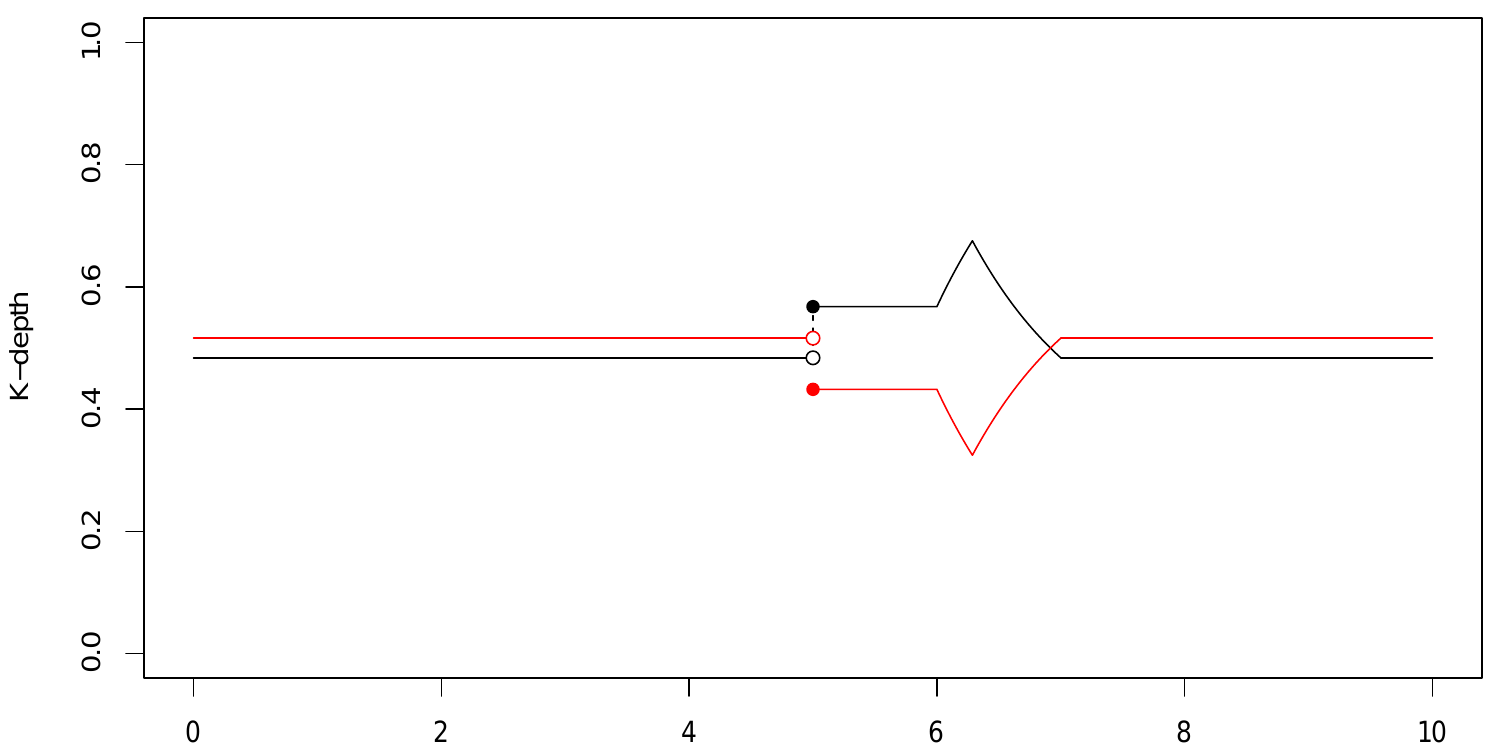}
      \caption{$K$-depth and its complement $1-\de_K$.}
    \end{subfigure}
    \caption{The one-dimensional mixture Example~\ref{ex:one-dimens-disc}. The discontinuity of the
      translation-mass function produces a half-open effective support.}
    \label{fig:univariate}
  \end{figure}
\end{example}

\subsection{Radial monotonicity and \texorpdfstring{$s$}{s}-concavity}

Radial monotonicity considered next is naturally centred at a
maximiser $m$ of the translation-mass function. The coordinate origin
has no intrinsic role and may differ from $m$.

\begin{definition}
  The translation-mass function is \emph{radially nonincreasing about}
  $m\in\R^d$ if, for every $u\in\Sph^{d-1}$, the map
  $r\mapsto f_{K,\mu}(m+ru)$, $r\ge0$, is nonincreasing.
\end{definition}

If $K=-K$ and $\mu$ has a density symmetric about $m$ with convex upper
level sets, Anderson's theorem \cite{Anderson1955} implies radial nonincrease about $m$. Thus $m$ is a maximiser and the high-mass shift
sets are star-shaped about $m$. Anderson's theorem by itself does
not yield strict inequality. Bo\v{c}inec and Nagy
\cite{BocinecNagy2024} give necessary and sufficient conditions for
equality, together with tractable sufficient conditions for strict
inequality at a prescribed nonzero shift. If these strictness conditions
hold for every nonzero shift, then $m$ is the unique maximiser, and
Proposition~\ref{prop:effective-support} gives
$S_K(\mu)=\Int K+m$.

For $a,b\ge0$, $\lambda\in(0,1)$, and $s\in\R\setminus\{0\}$, let
$M_s(a,b;\lambda) :=\bigl((1-\lambda)a^s+\lambda b^s\bigr)^{1/s}$,
with the standard continuous extensions
$M_0(a,b;\lambda)=a^{1-\lambda}b^\lambda$ and
$M_{-\infty}(a,b;\lambda)=\min(a,b)$. A density $q$ on a convex set
$S$ is $s$-concave if
\[
  q((1-\lambda)x+\lambda y) \ge M_s(q(x),q(y);\lambda) \qquad x, y \in S, \lambda \in (0,1).
\]

\begin{proposition}
  Suppose that $\mu$ has an $s$-concave density on $\R^d$ with
  $s\ge-1/d$. Then $f_{K,\mu}$ is $t$-concave, where $t=s/(1+ds)$ if
  $s>-1/d$, and $t=-\infty$ if $s=-1/d$. In particular, $f_{K,\mu}$ is
  quasi-concave and every set $L^>_\beta$ is convex. For $s=0$, the
  translation-mass function is log-concave.
\end{proposition}
\begin{pf}
  This is the integral form of the Borell--Brascamp--Lieb inequality;
  see \cite[Chapter~3]{DharmadhikariJoagDev1988}. The endpoint
  $s=-1/d$ is the limiting quasi-concave case, rather than the
  undefined quotient $s/(1+ds)$.
\end{pf}

\begin{example}[Gaussian contours]
  \label{ex:gaussian}
  Let $\mu$ be the centred bivariate Gaussian law with unit marginal
  variances and correlation $-0.6$. Figure~\ref{fig:gaussian} compares
  a triangular testing body with a Euclidean ball of radius $2$. The
  dashed curves are density contours of $\mu$. Log-concavity
  guarantees convex translation-mass superlevel sets, but the depth
  contours retain the geometry of the testing body.

\begin{figure}[ht]
\centering
\begin{subfigure}{0.48\textwidth}
 \centering\includegraphics[width=\linewidth]{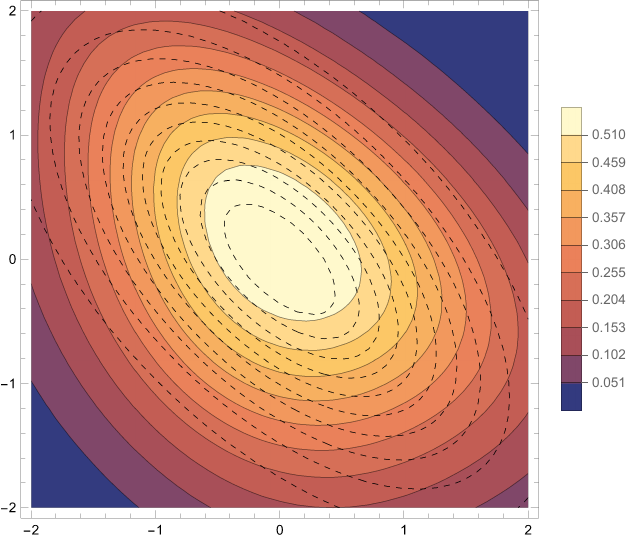}
 \caption{Translation mass, triangular $K$.}
\end{subfigure}\hfill
\begin{subfigure}{0.48\textwidth}
 \centering\includegraphics[width=\linewidth]{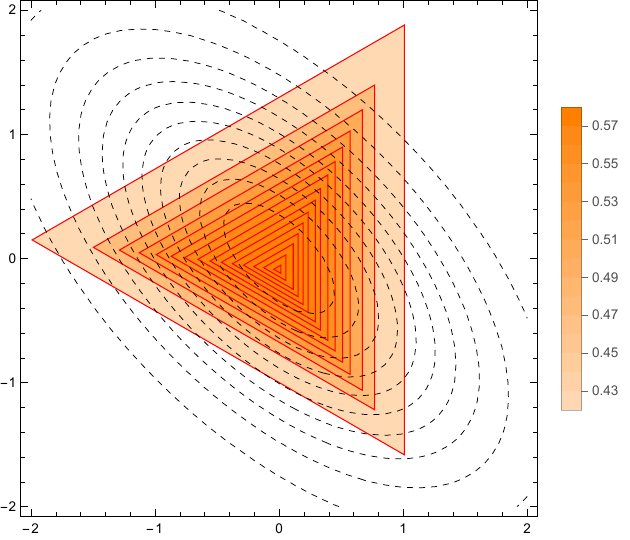}
 \caption{Depth, triangular $K$.}
\end{subfigure}

\medskip
\begin{subfigure}{0.48\textwidth}
 \centering\includegraphics[width=\linewidth]{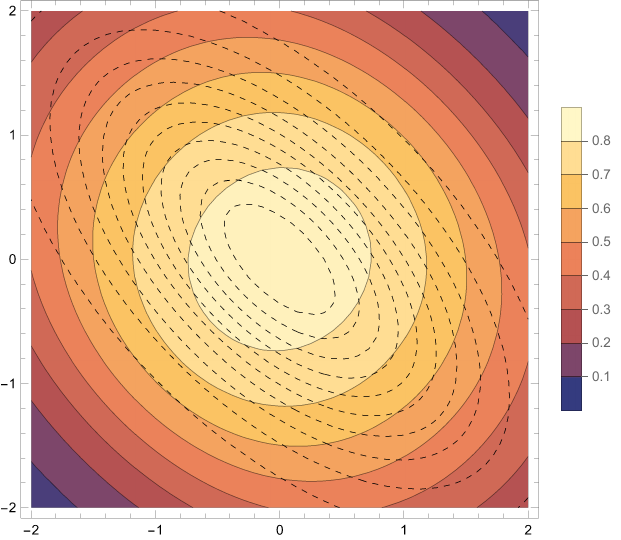}
 \caption{Translation mass, ball $K$.}
\end{subfigure}\hfill
\begin{subfigure}{0.48\textwidth}
 \centering\includegraphics[width=\linewidth]{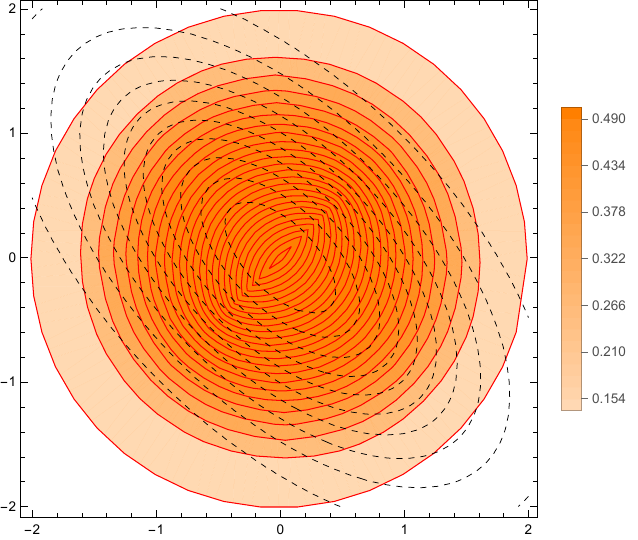}
 \caption{Depth, ball $K$.}
\end{subfigure}
\caption{Example~\ref{sec:notat-texorpdfstr-hu}: Translation-mass functions and $K$-depth contours for the same
Gaussian distribution and two testing bodies.}
\label{fig:gaussian}
\end{figure}
\end{example}

\section{Geometry of the depth regions}

\subsection{Dimension and generating bodies} \label{sec: dimension}

\begin{proposition}
  Let $\alpha\in(1-\beta_{K,\mu},1)$ and put $U=\cl
  L^>_{1-\alpha}$. Then $U$ is nonempty and compact, and
  $\Int D_\alpha^K(\mu) =\{y:\check U+y\subset\Int K\}$. Consequently,
  $D_\alpha^K(\mu)$ is full-dimensional if and only if there is a
  translate of $\check U$ contained in $\Int K$.
\end{proposition}
\begin{pf}
  The threshold $1-\alpha$ is positive and below $\beta_{K,\mu}$, so
  $L^>_{1-\alpha}$ is nonempty and bounded by
  Lemma~\ref{lem:f-basic}. Let $D=K\ominus U$. If $y\in\Int D$, then
  $y+rB\subset D$ for some $r>0$, and therefore
  $\check U+y+rB\subset K$. Hence $\check U+y\subset\Int K$.
  Conversely, if the compact set $\check U+y$ is contained in
  $\Int K$, it has positive distance from $K^c$. Thus
  $\check U+y+rB\subset K$ for some $r>0$, and $y+rB\subset D$.
\end{pf}

\begin{corollary}
  If $\mu$ is $K$-smooth and $\de_K(y;\mu)>\alpha$, then
  $y\in\Int D_\alpha^K(\mu)$. In particular, a region that contains a
  point of depth strictly larger than its level is full-dimensional.
\end{corollary}
\begin{pf}
  Continuity of the depth gives a neighbourhood on which the depth
  remains larger than $\alpha$.
\end{pf}

A convex body $K$ is called \emph{generating} if every nonempty
intersection of translates of $K$ is a Minkowski summand of
$K$. Euclidean balls (and ellipsoids) are generating in every
dimension, and every planar convex body is generating, see
\cite{Polovinkin1996,Schneider2014}. Further examples are given in
\cite{MR1753492,MR2642723}. If $K$ is a generating body, then any
nonempty $D_\alpha^K(\mu)$ is a summand of $K$.

\begin{proposition}
  Suppose that $f_{K,\mu}(x)>1-\alpha$ for some $x$,
  $D_\alpha^K(\mu)\ne\varnothing$, and that $K$ is a Euclidean
  ball. Then every depth region is either $\R^d$, empty, a singleton,
  or full-dimensional.
\end{proposition}
\begin{pf}
  By Proposition~\ref{prop:baseline}, $D=D_\alpha^K(\mu)$ is distinct
  from the whole space. Assume that $D$ is nonempty. If $D$ has any
  non-singleton face, then the sum also has such a face, which is
  impossible since $K$ is a ball. This excludes the case of $D$ having
  dimension between $1$ and $d-1$, since such sets always have a
  nontrivial flat face. Thus $D$ is a singleton or is
  full-dimensional.
\end{pf}

This result holds also in all cases when $K$ is generating and
strictly convex.

\subsection{Active translates and normal cones}

For a convex set $C$ and $y\in C$, write
$N_C(y):=\{u:\langle u,z-y\rangle\le0\text{ for all }z\in C\}$ for its
normal cone. A convex body is \emph{smooth} if every boundary point
has a unique supporting hyperplane, equivalently, the normal cone is
one-dimensional for all $y\in\partial C$. For convex body $K$,
\begin{equation}
  \label{eq: support function}
  h_K(u):=\sup\{\langle x,u\rangle: x\in K\},\quad u\in\R^d,
\end{equation}
denotes its \emph{support function}. 

\begin{theorem}
\label{thm:normal-cone}
  Assume that $K$ is smooth, let $\alpha\in(1-\beta_{K,\mu},1)$, and
  suppose that $D=D_\alpha^K(\mu)$ has nonempty interior. Put
  $U:=\cl L^>_{1-\alpha}$ and
  \begin{displaymath}
    A_\alpha(y):=\{x\in U:y\in\partial(K+x)\},
    \qquad y\in\partial D.
  \end{displaymath}
  Then $A_\alpha(y)$ is nonempty and compact, and
  \begin{equation}
    N_D(y)
    =\cl\operatorname{cone}
    \{n_K(y-x):x\in A_\alpha(y)\},
    \label{eq:normal-cone}
  \end{equation}
  where $n_K(p)$ is the unique outer unit normal of $K$ at $p$.
  Consequently, $D$ has a unique outer normal at $y$ if and only if
  $n_K(y-x)=n_K(y-x')$ for all $x,x'\in A_\alpha(y)$.  A sufficient
  condition is that every active set $A_\alpha(y)$ is a singleton.
\end{theorem}
\begin{pf}
  The representation $D=\bigcap_{x\in U}(K+x)$ follows by passing from
  the strict level set to its closure. If $A_\alpha(y)$ were empty,
  then $y$ would lie in the interior of every $K+x$. Compactness of
  $U$ and continuity of $x\mapsto\dist(y,(K+x)^c)$ would give a
  uniform positive margin, making $y$ an interior point of $D$. Thus
  the active set is nonempty; it is closed in the compact set $U$.

  To justify the normal-cone formula, use the support-function
  representation of the translates. Define the finite convex function
  \begin{displaymath}
    G(z):=\max_{(x,u)\in U\times\Sph^{d-1}}
    \{\langle u,z-x\rangle-h_K(u)\}.
  \end{displaymath}
  Since $U\times\Sph^{d-1}$ is compact, the maximum is attained, and
  since $\Int D\ne\varnothing$ there is $z_0$ with $G(z_0)<0$. Then
  $D=\{z:G(z)\le0\}$ and $G(y)=0$ for all $y\in\partial D$.  The
  compact maximum-function subdifferential formula
  \cite[Theorem~10.31]{RockafellarWets1998} gives
  $\partial G(y) =\operatorname{conv}\{u: \ \langle
  u,y-x\rangle=h_K(u) \; \text{for some}\; x\in U\}$. By the
  normal-cone formula for a convex sublevel set under strict
  feasibility \cite[Theorem~23.7]{Rockafellar1970},
  $N_D(y)=\operatorname{cone}\partial G(y)$. Equality in the support
  inequality is equivalent to $y\in\partial(K+x)$, and smoothness of
  $K$ makes the corresponding unit vector unique, namely
  $u=n_K(y-x)$. This proves \eqref{eq:normal-cone}. The final
  equivalence follows because a normal cone is a ray precisely when
  all its generating unit vectors agree.
\end{pf}

The following result concerns testing bodies whose boundary is $C^2$. 

\begin{theorem}
  \label{thm:c2-envelope}
  In the setup of Theorem~\ref{thm:normal-cone}, assume in addition
  that $K=\{z:\rho(z)\le0\}$ for a $C^2$ function $\rho:\R^d\to\R$
  such that $\nabla\rho(z)\ne0$ for all $z\in\partial K$. Define
  $G_\alpha(y):=\max_{x\in U}\rho(y-x)$ for $y\in\R^d$. Then
  \begin{equation}
    D_\alpha^K(\mu)=\{y:G_\alpha(y)\le0\}.
    \label{eq:c2-envelope-sublevel}
  \end{equation}
  If $G_\alpha$ is $C^2$ in a neighbourhood of
  $\partial D_\alpha^K(\mu)$ and $\nabla G_\alpha(y)\ne0$ for all
  $y\in\partial D_\alpha^K(\mu)$, then $D_\alpha^K(\mu)$ has $C^2$
  boundary.
\end{theorem}
\begin{pf}
  By Proposition~\ref{prop:erosion},
  $D_\alpha^K(\mu)=\bigcap_{x\in L^>_{1-\alpha}}(K+x)$. Because $K$ is
  closed, passing from the indexing set to its closure leaves the
  intersection unchanged. Hence
  $D_\alpha^K(\mu)=\bigcap_{x\in U}(K+x)$. The set $U$ is compact, so
  the maximum defining $G_\alpha$ is attained and $G_\alpha$ is
  continuous. A point $y\in\R^d$ belongs to $D_\alpha^K(\mu)$ if and
  only if it belongs to $K+x$ for every $x\in U$. Since $\rho$ is a
  defining function for $K$, this condition holds if and only if
  $\rho(y-x)\le0$ for every $x\in U$. Because the maximum over $U$ is
  attained, the latter condition is equivalent to
  $G_\alpha(y)\le0$. This proves \eqref{eq:c2-envelope-sublevel}.

  Let $y_0\in\partial D_\alpha^K(\mu)$. Since the depth region is
  closed, $y_0\in D_\alpha^K(\mu)$, so that
  $G_\alpha(y_0)\le0$. If $G_\alpha(y_0)<0$, continuity of $G_\alpha$
  would give $G_\alpha<0$ throughout a neighbourhood of $y_0$. By
  \eqref{eq:c2-envelope-sublevel}, this neighbourhood would be
  contained in $D_\alpha^K(\mu)$, contradicting
  $y_0\in\partial D_\alpha^K(\mu)$. Thus $G_\alpha(y_0)=0$.

  By assumption, $G_\alpha$ is $C^2$ near $y_0$ and
  $\nabla G_\alpha(y_0)\ne0$. The implicit function theorem shows that
  the level set $\{G_\alpha=0\}$ is a $C^2$ surface in a neighbourhood
  of $y_0$. In view of $D_\alpha^K(\mu)=\{G_\alpha\le0\}$ and the
  nonvanishing gradient, this level set is locally exactly the
  boundary of $D_\alpha^K(\mu)$. Since $y_0$ was arbitrary, the whole
  boundary is $C^2$.
\end{pf}

\subsection{Changing the testing body}

There is no general set-inclusion monotonicity in $K$: enlarging $K$
enlarges each translate but also enlarges the set of shifts whose mass
exceeds the threshold. The following elementary bound is sometimes
useful.

\begin{proposition}
  Let $K_2=K_1+M$, where $M\in\Kc^d$ and $0\in M$. For
  $i=1,2$, let $L_i:=\{x:\mu(K_i+x)>1-\alpha\}$.
  Then $L_1\subset L_2$ and $D_\alpha^{K_2}(\mu)
    \subset\bigcap_{x\in L_1}(K_1+M+x)
    =(K_1+M)\ominus L_1$.
\end{proposition}
\begin{pf}
  The assumption $0\in M$ gives $K_1\subset K_2$, hence
  $\mu(K_1+x)\le\mu(K_2+x)$ and $L_1\subset L_2$. Intersecting over
  the larger index set gives the claim.
\end{pf}

\begin{theorem}
  \label{thm:K-stability}
  Let $K_n,K\in\Kc^d_\circ$ and suppose that $d_H(K_n,K)\to 0$. Assume
  that $\mu$ is $K$-smooth. Then
  \begin{equation}
    \sup_{y\in\R^d}
    \left| \de_{K_n}(y;\mu)-\de_K(y;\mu)\right| \to 0 .
    \label{eq:K-stability-depth}
  \end{equation}
  If, in addition, $\alpha\in(0,1)$ is a regular level for $K$ in the
  sense that $D_\alpha^K(\mu)$ is nonempty and compact and
  \begin{displaymath}
    \lim_{\eta\downarrow0}
    d_H(D_{\alpha+\eta}^K(\mu),D_\alpha^K(\mu))=0,
    \qquad
    \lim_{\eta\downarrow0}
    d_H(D_{\alpha-\eta}^K(\mu),D_\alpha^K(\mu))=0,
  \end{displaymath}
  then $d_H(D_\alpha^{K_n}(\mu),D_\alpha^K(\mu))\to 0$. 
\end{theorem}
\begin{pf}
  For $r>0$, put $b_K(r)=\sup_{x\in\R^d}\mu((\partial K+x)^r)$, where
  $(\partial K+x)^r$ is the closed $r$-parallel enlargement of
  $\partial K+x$. By the uniform boundary-band argument in the proof
  of Lemma~\ref{lem:uniform-translates}, $b_K(r)\to0$ as
  $r\downarrow0$.  We shall use the following elementary consequence
  of Hausdorff convergence for convex bodies. If $C$ and $L$ are
  convex bodies and $d_H(C,L)\le r$, then $C^{-r}\subset L$ and
  $L^{-r}\subset C$.
  Thus points of $C\setminus L$ lie within distance $r$ of
  $\partial C$, and points of $L\setminus C$ lie within distance $r$
  of $\partial L$.

  For the function $M_K(y;\mu)$ defined in \eqref{eq:depth}, we prove
  the deterministic estimate
  \begin{equation}
    \sup_y|M_{K_n}(y;\mu)-M_K(y;\mu)|
    \le b_K(2r_n+\gamma),
    \qquad r_n=d_H(K_n,K),
    \label{eq:testing-body-M-bound}
  \end{equation}
  for every fixed $\gamma>0$ and all sufficiently large $n$. This implies
  \eqref{eq:K-stability-depth}, since $\de_K=1-M_K$.

  Fix a convex body $L$ with $d_H(L,K)=r$, and fix $y\in\R^d$. We
  first compare $M_L(y;\mu)$ with $M_K(y;\mu)$. Let $x$ satisfy
  $y\notin L+x$, and write $q=y-x$. If $q\notin K$, then $x$ is also
  admissible for $M_K(y;\mu)$. If $q\in K\setminus L$, then the
  convexity of $K$ and the Hausdorff bound imply that $q$ lies within
  distance $r$ of $\partial K$. Hence we may choose $q'$ outside $K$
  with $\|q-q'\|\le r+\gamma$ and put $x'=y-q'$. Then $y\notin K+x'$
  and $\|x'-x\|\le r+\gamma$. Since $L\subset K+rB$, we have
  $L+x\subset (K+x')+(2r+\gamma)B$. Therefore
  $\mu(L+x) \le \mu(K+x')+b_K(2r+\gamma) \le
  M_K(y;\mu)+b_K(2r+\gamma)$. Taking the supremum over all admissible
  $x$ gives $M_L(y;\mu)\le M_K(y;\mu)+b_K(2r+\gamma)$.
    
  The reverse inequality is analogous but uses the same boundary band
  of $K$. Let $x$ satisfy $y\notin K+x$ and put $q=y-x$. If
  $q\notin L$, then $x$ is admissible for $M_L(y;\mu)$. If
  $q\in L\setminus K$, then $q$ is within distance $r$ of
  $\partial L$, so we can find $q'$ outside $L$ with
  $\|q-q'\|\le r+\gamma$ and put $x'=y-q'$. Then $y\notin L+x'$ and
  $\|x'-x\|\le r+\gamma$. Since $d_H(K,L+x'-x)\le2r+\gamma$, every
  point of $K\setminus(L+x'-x)$ lies within distance $2r+\gamma$ of
  $\partial K$.  Consequently,
  \begin{displaymath}
    \mu(K+x)
    \le \mu(L+x')+b_K(2r+\gamma)
    \le M_L(y;\mu)+b_K(2r+\gamma).
  \end{displaymath}
  Taking suprema gives the opposite bound. Applying this with $L=K_n$
  proves \eqref{eq:testing-body-M-bound} and it suffices to let
  $\gamma\downarrow0$.  Now we pass from depths to regions. Let
  $\Delta_n=\sup_y|\de_{K_n}(y;\mu)-\de_K(y;\mu)|$. For every
  $\eta>\Delta_n$, we have
  $D_{\alpha+\eta}^K(\mu)\subset D_\alpha^{K_n}(\mu) \subset
  D_{\alpha-\eta}^K(\mu)$. Choose $\eta_n\downarrow0$ such that
  $\Delta_n<\eta_n$ eventually. The regular-level assumption and the
  displayed sandwich imply the Hausdorff convergence of
  $D_\alpha^{K_n}(\mu)$ to $D_\alpha^K(\mu)$.
\end{pf}

\section{Large testing bodies and tangent-cone depth}

\subsection{Dilation limit}
\label{sec:dilation-limit}

For $p\in\partial K$, define the tangent cone of $K$ at $p$ by
\begin{displaymath}
  T_K(p):=\cl\bigcup_{s>0}s(K-p).
\end{displaymath}
For $y\in\R^d$, put
\begin{align*}
  M_K^\infty(y;\mu)
  &:=\sup\{\mu(z+T_K(p)):
    p\in\partial K,\ z\in\R^d,\ y\notin z+T_K(p)\},
  \\
  \de_K^\infty(y;\mu)
  &:=1-M_K^\infty(y;\mu).
\end{align*}
We call $\de_K^\infty$ the \emph{tangent-cone depth generated by $K$}.
Since every tangent cone is contained in one of its supporting
halfspaces, we get $\HD(y;\mu)\le\de_K^\infty(y;\mu)$.

\begin{proposition}
  \label{prop:dilation-monotonicity}
  If $0<s\le t$, then $\de_{tK}(y;\mu)\le\de_{sK}(y;\mu)$ for all $y\in\R^d$.
\end{proposition}
\begin{pf}
  Let $sK+x$ be a translate missing $y$. Strict separation gives a
  unit vector $u$ such that $h_{sK+x}(u)<\langle y,u\rangle$ (here,
  $h_K$ is the support function of $K$ from~\eqref{eq: support
    function}).  Choose $p\in K$ with $\langle p,u\rangle=h_K(u)$ and
  set $z=x+sp$. Then $sK+x=z+s(K-p)\subset z+t(K-p)$, because $K-p$ is
  convex and contains the origin. The larger set is a translate of
  $tK$. Moreover, $K-p\subset\{v:\langle v,u\rangle\le0\}$, so every
  point of $z+t(K-p)$ has scalar product with $u$ at most
  $\langle z,u\rangle=h_{sK+x}(u)<\langle y,u\rangle$. Thus the larger
  translate also misses $y$ and has mass at least $\mu(sK+x)$. Taking
  suprema over the admissible translates proves our claim.
\end{pf}


\begin{theorem}
  \label{thm:dilation}
  Assume that
  \begin{equation}
    \mu(\partial(z+T_K(p)))=0
    \quad\text{for every }p\in\partial K\text{ and }z\in\R^d.
    \label{eq:cone-boundary-null}
  \end{equation}
  Then $\de_{tK}(y;\mu)\to \de_K^\infty(y;\mu)$ as $t\to\infty$ for
  every $y\in\R^d$.
\end{theorem}
\begin{pf}
  Recall that $M_{tK}(y;\mu)=\sup\{\mu(tK+x):y\notin tK+x\}$. Fix
  $p\in\partial K$ and $z\in\R^d$ such that $y\notin z+T_K(p)$. Since
  $z+t(K-p)\subset z+T_K(p)$, the set $z+t(K-p)$ is admissible for
  $M_{tK}(y;\mu)$. The sets $z+t(K-p)$ increase with $t$, and their union is
  dense in $z+T_K(p)$. The part of the tangent cone not covered by the
  union lies on its boundary. Hence \eqref{eq:cone-boundary-null} and
  continuity from below give
  $\liminf_{t\to\infty}M_{tK}(y;\mu)\ge\mu(z+T_K(p))$. Taking the supremum
  over admissible $p,z$ yields
  \begin{equation}
    \liminf_{t\to\infty}M_{tK}(y;\mu)\ge M_K^\infty(y;\mu).
    \label{eq:dilation-liminf}
  \end{equation}
  Conversely, let $y\notin tK+x$. Strict separation gives
  $u\in\Sph^{d-1}$ such that
  $\sup_{v\in tK+x}\langle v,u\rangle<\langle y,u\rangle$. Choose
  $p\in K$ satisfying $\langle p,u\rangle=h_K(u)$ and put
  $z=tp+x$. Then $tK+x=z+t(K-p)\subset z+T_K(p)$.  Moreover,
  $T_K(p)\subset\{w:\langle w,u\rangle\le0\}$, so $z+T_K(p)$ lies in
  the supporting halfspace of $tK+x$ and still misses $y$. Therefore
  $\mu(tK+x)\le M_K^\infty(y;\mu)$. Taking the supremum gives
  $M_t(y)\le M_K^\infty(y;\mu)$ for every $t$, which together with
  \eqref{eq:dilation-liminf} proves the claim.
\end{pf}

\begin{corollary}
  \label{cor:smooth-dilation}
  If $K$ is smooth and $\mu$ assigns zero mass to every affine
  hyperplane, then $\de_{tK}(y;\mu)\to \HD(y;\mu)$ for every $y\in\R^d$.
\end{corollary}
\begin{pf}
  At every boundary point of a smooth convex body the tangent cone is
  a closed halfspace. Conversely, every direction occurs as an outer
  normal at a support point of $K$. Thus the tangent-cone testing
  family is exactly the family of closed halfspaces, and
  $\de_K^\infty=\HD$. The boundary-null condition in
  Theorem~\ref{thm:dilation} becomes the hyperplane-null condition.
\end{pf}

Together with Proposition~\ref{prop:dilation-monotonicity},
Corollary~\ref{cor:smooth-dilation} shows that for a smooth testing body
and a hyperplane-null measure we have  $\de_{tK}(y;\mu)\downarrow\HD(y;\mu)$ as $t\to\infty$.

\subsection{A quantitative halfspace approximation}
\label{subsec:quantitative-dilation}

The qualitative limit in Theorem~\ref{thm:dilation} can be made
uniform, with an explicit error, if the boundary of the testing body
admits uniformly sized inner tangent balls.

\begin{definition}
  \label{def:rolling-ball}
  The testing body $K$ satisfies the interior rolling-ball condition
  with radius $\rho>0$ if, for every $u\in\Sph^{d-1}$, there is a
  support point $p(u)\in\partial K$ such that
  \begin{equation}
    \langle p(u),u\rangle=h_K(u),
    \qquad
    p(u)-\rho u+B_\rho\subset K.
    \label{eq:rolling-ball}
  \end{equation}
  In particular, every convex body with $C^2$ boundary satisfies it
  for some $\rho>0$.
\end{definition}

For a probability measure $\mu$, define its one-sided slab modulus by
\begin{displaymath}
  \omega_\mu(r)
  :=\sup_{u\in\Sph^{d-1}, a\in\R}
  \mu(\{z:a-r<\langle z,u\rangle\le a\}),
  \qquad r>0.
\end{displaymath}

\begin{lemma}
  \label{lem:slab-modulus}
  If $\mu$ has compact support and assigns zero mass to every affine
  hyperplane, then $\omega_\mu(r)\to0$ as $r\downarrow0$.
\end{lemma}
\begin{pf}
  Suppose otherwise. There are $r_n\downarrow0$, unit vectors $u_n$,
  real numbers $a_n$, and $\eps>0$ such that the corresponding slabs
  have mass at least $\eps$. Compactness of the support implies that
  the relevant $a_n$ are bounded. After passing to a subsequence,
  $u_n$ and $a_n$ converge to some $u$ and $a$. For every $\eta>0$,
  the slabs for all sufficiently large $n$ are contained, on the
  support of $\mu$, in $\{z:|\langle z,u\rangle-a|\le\eta\}$. Hence
  each of these closed $\eta$-neighbourhoods has mass at least
  $\eps$. Continuity from above, as $\eta\downarrow0$, gives positive
  mass to the hyperplane $\{z:\langle z,u\rangle=a\}$, a
  contradiction.
\end{pf}

\begin{theorem}
  \label{thm:quantitative-dilation}
  Assume that $K$ satisfies \eqref{eq:rolling-ball} with radius
  $\rho$, that $\supp\mu\subset c+B_R$, and that $\mu$ assigns zero
  mass to every affine hyperplane. For $t\rho>R$, put
  \begin{displaymath}
    \delta_t:=R^2/(2(t\rho-R)).
  \end{displaymath}
  Then, simultaneously for all $y\in\R^d$,
  \begin{equation}
    0\le \de_{tK}(y;\mu)-\HD(y;\mu)
    \le\omega_\mu(\delta_t).
    \label{eq:uniform-dilation-bound}
  \end{equation}
  Consequently,
  $\sup_{y\in\R^d} |\de_{tK}(y;\mu)-\HD(y;\mu)|
  \le\omega_\mu(\delta_t)\to 0$.
    
  Suppose further that $\mu$ has a density $q$ with
  $\|q\|_\infty\le L$. Then
  \begin{equation} 
    \sup_{y \in \R^d} |\de_{tK}(y;\mu)-\HD(y;\mu)|
    \le
    \frac{L\kappa_{d-1}R^{d+1}}{2(t\rho-R)}.
    \label{eq:bounded-density-rate}
  \end{equation}
  For $t\ge2R/\rho$, the right-hand side is at most
  $L\kappa_{d-1}R^{d+1}/(t\rho)$.
\end{theorem}
\begin{pf}
  For $u\in\Sph^{d-1}$ and $a\in\R$, write
  $H(u,a)=\{z:\langle z,u\rangle\le a\}$. The hyperplane-null
  assumption gives the complementary representation
  \begin{equation}
    1-\HD(y;\mu)
    =\sup_{u\in\Sph^{d-1}, a<\langle y,u\rangle} \mu(H(u,a)).
    \label{eq:halfspace-complementary}
  \end{equation}
  Indeed, the complement of a halfspace containing $y$ can be
  approximated from within by closed halfspaces missing $y$, and the
  boundary hyperplane carries no mass.

  Fix $u$ and $a<\langle y,u\rangle$. Halfspaces of zero mass do not
  affect the supremum in \eqref{eq:halfspace-complementary}. If
  $a>\langle c,u\rangle+R$, replace $a$ by $\langle c,u\rangle+R$;
  this does not change the mass and preserves the strict inequality
  $a<\langle y,u\rangle$. We may therefore assume
  $|a-\langle c,u\rangle|\le R$. Let $z_0=c+(a-\langle c,u\rangle)u$
  and choose the support point $p=p(u)$ from
  \eqref{eq:rolling-ball}. The set $A_t=z_0+t(K-p)$ is a translate of
  $tK$. Since $K-p$ lies in $\{v:\langle v,u\rangle\le0\}$, the set
  $A_t$ is contained in $H(u,a)$ and therefore misses $y$. The
  rolling-ball condition gives
  \begin{equation}
    z_0-t\rho u+B_{t\rho}\subset A_t.
    \label{eq:tangent-ball-inclusion}
  \end{equation}
  Take $q\in(c+B_R)\cap H(u,a)$ and write $q-z_0=v-su$,
  $v\in u^\perp$, and $s=a-\langle q,u\rangle$. Then $|v|\le R$ and
  $0\le s\le2R$. If $s\ge\delta_t$, then
  \begin{displaymath}
    |q-(z_0-t\rho u)|^2 =|v|^2+(t\rho-s)^2 \le(t\rho)^2+R^2-2(t\rho-R)s
    \le(t\rho)^2.
  \end{displaymath}
  Thus $q$ belongs to the ball in
  \eqref{eq:tangent-ball-inclusion}. The part of $H(u,a)\cap\supp\mu$
  not captured by $A_t$ is consequently contained in
  $\{z:a-\delta_t<\langle z,u\rangle\le a\}$. It follows that
  $\mu(A_t)\ge\mu(H(u,a))-\omega_\mu(\delta_t)$.  Taking the supremum
  over the halfspaces in~\eqref{eq:halfspace-complementary} gives
  $\sup_{x:\,y\notin tK+x}\mu(tK+x)
  \ge1-\HD(y;\mu)-\omega_\mu(\delta_t)$.

  Conversely, every translate of $tK$ missing $y$ is contained, by
  strict separation, in a halfspace missing $y$. Its mass is therefore
  at most $1-\HD(y;\mu)$. Translating these two inequalities from
  missed-set mass to depth proves
  \eqref{eq:uniform-dilation-bound}. The uniform convergence follows
  from Lemma~\ref{lem:slab-modulus}.

  With a bounded density $q$, the intersection of a slab of width $r$
  with $c+B_R$ has volume at most $\kappa_{d-1}R^{d-1}r$. Hence
  $\omega_\mu(r)\le L\kappa_{d-1}R^{d-1}r$, and it remains to
  substitute $r=\delta_t$. 
\end{pf}

For the next consequences, write $D_{\alpha,t}=D_\alpha^{tK}(\mu)$,
$\eps_t=\omega_\mu(\delta_t)$, and
$\alpha_{\rm H}^*(\mu)=\sup_y\HD(y;\mu)$.

\begin{corollary}
  Maintain the assumptions of Theorem~\ref{thm:quantitative-dilation},
  assume that $0<\alpha<\alpha_{\rm H}^*(\mu)$ and
  $0<\eps_t<\alpha$. Then
  \begin{equation}
    D_\alpha^{\mathrm{HD}}(\mu)\subset D_{\alpha,t}
    \subset D_{\alpha-\eps_t}^{\mathrm{HD}}(\mu).
    \label{eq:quantitative-region-sandwich}
  \end{equation}
  Assume in addition that $\HD(\,\cdot\,;\mu)$ is $C^1$ on a
  neighbourhood of the compact strip
  $\{y:\alpha-\eta\le\HD(y;\mu)\le\alpha\}$ for some $\eta>0$, and
  that
  $m_{\alpha,\eta} :=\inf_{\alpha-\eta\le\HD(y;\mu)\le\alpha}
  \|\nabla\HD(y;\mu)\|>0$.  When $\eps_t<\eta$,
  \begin{displaymath}
    d_H(D_{\alpha,t},D_\alpha^{\mathrm{HD}}(\mu))
    \le\frac{\eps_t}{m_{\alpha,\eta}}.
  \end{displaymath}
\end{corollary}
\begin{pf}
  The sandwich~\eqref{eq:quantitative-region-sandwich} follows
  immediately from \eqref{eq:uniform-dilation-bound}. To obtain the
  Hausdorff estimate, consider on the strip the vector field
  $V(y)=\nabla\HD(y;\mu)/\|\nabla\HD(y;\mu)\|^2$.  Along every
  integral curve of this field, the halfspace depth increases at unit
  speed. A point of $D_{\alpha-\eps_t}^{\mathrm{HD}}(\mu)$ outside
  $D_\alpha^{\mathrm{HD}}(\mu)$ therefore reaches
  $D_\alpha^{\mathrm{HD}}(\mu)$ after time at most $\eps_t$, while its
  speed is at most $1/m_{\alpha,\eta}$. Thus every point of the outer
  set in \eqref{eq:quantitative-region-sandwich} is within
  $\eps_t/m_{\alpha,\eta}$ of $D_\alpha^{\mathrm{HD}}(\mu)$.
\end{pf}

Recall that the Tukey median of $\mu$ is the set of maximisers of the
halfspace depth of $\mu$.

\begin{corollary}
  Under the assumptions of Theorem~\ref{thm:quantitative-dilation},
  \begin{equation}
    0\le\alpha_{tK}^*(\mu)-\alpha_{\rm H}^*(\mu)
    \le\eps_t.
    \label{eq:maximal-depth-rate}
  \end{equation}
  Suppose, moreover, that the Tukey median is the single point $y_0$
  and that, in a neighbourhood of $y_0$,
  \begin{equation}
    \alpha_{\rm H}^*(\mu)-\HD(y;\mu)
    \ge c\|y-y_0\|^p
    \label{eq:tukey-median-growth}
  \end{equation}
  for some $c>0$ and $p\ge1$. If $y_t$ is any deepest point for the
  $tK$-depth, then $\|y_t-y_0\|\le(\eps_t/c)^{1/p}$ for all
  sufficiently large $t$.
\end{corollary}
\begin{pf}
  The uniform depth bound gives \eqref{eq:maximal-depth-rate}. It also
  gives $\HD(y_t;\mu)\ge\alpha_{\rm H}^*(\mu)-\eps_t$. Uniform
  convergence and uniqueness of the Tukey median place $y_t$ in the
  neighbourhood where \eqref{eq:tukey-median-growth} holds for all
  large $t$. The claimed bound follows.
\end{pf}

The order $O(t^{-1})$ obtained in~\eqref{eq:bounded-density-rate}
cannot be improved in general, as can be seen in an example of
$K = B_1$ in dimension $d \geq 2$, and $\mu$ uniform on $B_1$.

\begin{proposition}
  \label{prop:rolling-helly}
  If $K$ satisfies the rolling-ball condition \eqref{eq:rolling-ball},
  then $h_T(K)=d+1$.
\end{proposition}
\begin{pf}
  The upper bound follows from ordinary Helly's theorem
  from~\eqref{eq:translative-helly-range}. For the reverse bound,
  choose unit vectors $u_1,\ldots,u_{d+1}$ that form the vertices of a
  regular simplex centred at the origin, so that
  $\sum_{i=1}^{d+1}u_i=0$, and $\langle u_i,u_j\rangle=-\frac1d$ when
  $i\ne j$. Put $H_i=\{x:\langle x,u_i\rangle\le-1\}$. The total
  intersection of the $H_i$ is empty, since summing their defining
  inequalities would give $0\le-(d+1)$.
  
  For each $i$, choose $p_i=p(u_i)$ from \eqref{eq:rolling-ball}, put
  $z_i=-u_i$, and consider $G_i(t)=z_i+t(K-p_i)$.  This is a translate
  of $tK$, it is contained in $H_i$, and it contains the tangent ball
  $z_i-t\rho u_i+B_{t\rho}$.  Fix $c>d$ and put $y_k=cu_k$. For
  $i\ne k$ we get $\langle y_k,u_i\rangle=-\frac cd<-1$, so $y_k$ lies
  strictly inside $H_i$. More explicitly, write $y_k-z_i=v-su_i$ with
  $v\in u_i^\perp$ and $s>0$. Then
  $\|y_k-(z_i-t\rho u_i)\|^2 =\|v\|^2+(t\rho-s)^2\le(t\rho)^2$ once
  $2t\rho s\ge \|v\|^2+s^2$. Hence $y_k$ belongs to the displayed
  tangent ball for all sufficiently large $t$. Since only finitely
  many pairs $(i,k)$ occur, one value of $t$ works for all of them. It
  follows that $y_k\in\bigcap_{i\ne k}G_i(t)$ for every $k$, whereas
  the total intersection of the $G_i(t)$ is empty because it is
  contained in the total intersection of the $H_i$. Thus the family of
  translates of $tK$ has Helly number at least $d+1$.  Positive
  homotheties do not change the translative Helly number, so
  $h_T(K)\ge d+1$.
\end{pf}

\begin{remark}
  Theorems~\ref{thm:sharp-centerpoint} and
  \ref{thm:quantitative-dilation} exhibit two complementary effects of
  the geometry of $K$. A body satisfying the rolling-ball condition
  has the classical worst-case centerpoint constant $1/(d+1)$ and
  converges quantitatively to halfspace depth. A parallelotope has the
  larger constant $1/2$, but its nonsmooth tangent cones can prevent
  convergence to halfspace depth, as the next example shows.
\end{remark}

\begin{example}[A nonsmooth limit different from halfspace depth]
  Let $K=[-1,1]^2$, $y=0$, and let $\mu$ be the average of two smooth
  probability densities supported in discs of radius $r<1/4$ centred
  at $a=(1,-1/2)$ and $b=(-1/2,1)$.  Both discs lie in the open
  halfspace $\{x_1+x_2>0\}$, so a closed halfspace missing the origin
  contains the whole support and $\HD(0;\mu)=0$.

  A translate of $tK$ is an axis-parallel square. If it met both
  support components, its projection on each coordinate axis would
  contain both a negative and a positive number, and therefore the
  square would contain the origin. Hence a translate missing the
  origin captures at most one component. Conversely, for all
  sufficiently large $t$, one component can be covered by a square
  missing the origin. Thus $\de_{tK}(0;\mu)=1/2$ for all sufficiently
  large $t$.
\end{example}

\section{Uniform distributions and covariograms}

\subsection{Uniform distribution on a convex body}
\label{sec:unif-distr-conv}

Let $A\in\Kc^d_\circ$ and let $\mu$ be uniform on $A$. Then
$f_{K,\mu}(x) = |A\cap(K+x)|/|A|$, which is the normalised
\emph{cross-covariogram} of $(A,K)$, see \cite{Schneider2014}. Since
$\mu$ is $K$-smooth, the depth is continuous; for an example see
Figure~\ref{fig:uniform-ellipse-pentagon}.

\begin{proposition}
  \label{prop:overlap-concavity}
  The function $x\mapsto |A\cap(K+x)|^{1/d}$ is concave on its
  support. Consequently, each strict superlevel set $L^>_\beta$ is
  convex, and for $\beta>0$ its closure is a compact convex set.
\end{proposition}
\begin{pf}
  For $x_t=(1-t)x_1+tx_2$, convexity of $A$ and $K$ gives
  $(1-t)(A\cap(K+x_1))+t(A\cap(K+x_2)) \subset A\cap(K+x_t)$. The
  Brunn--Minkowski inequality proves the claimed concavity. The rest
  follows from Proposition~\ref{prop:erosion}.
\end{pf}

Plateaux may occur when $K+x$ can move inside $A$ without losing
volume. Thus convexity of the overlap level sets does not by itself
guarantee strict convexity, uniqueness of active translates, or
smoothness of the depth boundary.

\begin{figure}[ht]
\centering
\begin{subfigure}{0.48\textwidth}
 \centering\includegraphics[width=\linewidth]{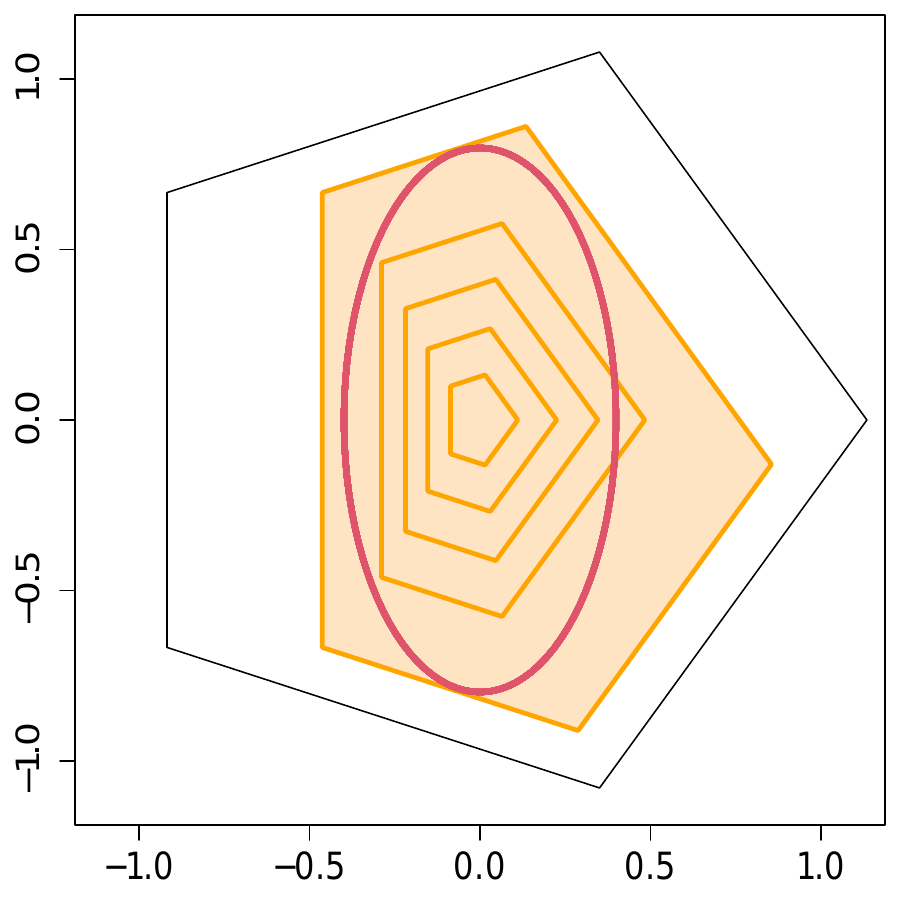}
 \caption{Nested depth regions.}
\end{subfigure}\hfill
\begin{subfigure}{0.48\textwidth}
 \centering\includegraphics[width=\linewidth]{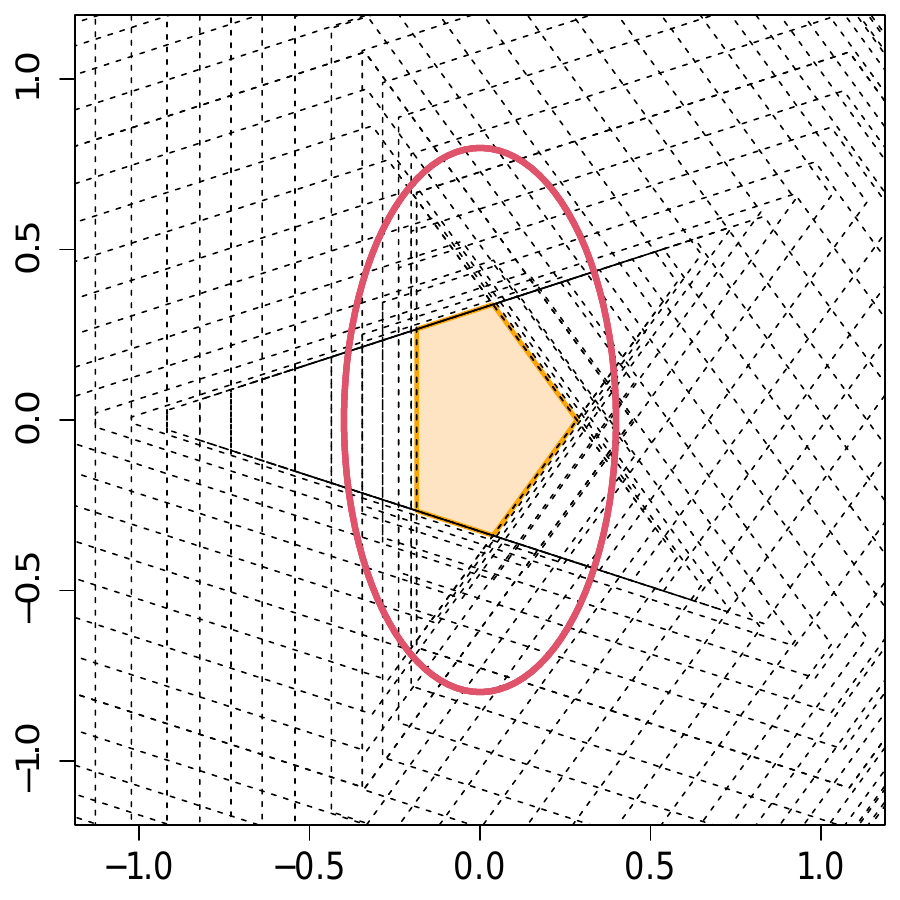}
 \caption{Active translates at one level.}
\end{subfigure}
\caption{Section~\ref{sec:unif-distr-conv}: Uniform distribution on an ellipse with a pentagonal testing
body. The orange regions are $K$-depth regions and the dashed curves in
the right panel are active translates.}
\label{fig:uniform-ellipse-pentagon}
\end{figure}

\subsection{Ball floating bodies and relative surface areas}

The terminology of Sch\"utt, Werner and Yalikun
\cite{SchuettWernerYalikun2025} agrees with the $K$-convexity used
above apart from the fact that they use $C$ instead of $K$. A nonempty
proper $C$-convex set is called a \emph{$C$-ball-convex set}; if it
has nonempty interior, it is a $C$-ball-convex body.  For such a body
$A$ and $\delta\ge0$, they define the $C$-ball floating body by
\begin{equation}
  A_\delta^C
  :=\bigcap_{x:\,|A\setminus(C+x)|\le\delta}(C+x).
  \label{eq:C-ball-floating-definition}
\end{equation}
When $C=B_R$, the set $A$ is said to be $R$-ball-convex; the body
$A_\delta^C$ is called the $R$-ball floating body and denoted by
$A_\delta^R$.  The translation-mass function is the normalised
cross-covariogram used in \cite{SchuettWernerYalikun2025}.


Werner and Yalikun \cite{WernerYalikun2026} define a weighted version.
Let $w:A\to(0,\infty)$ be continuous and put
\[
  W=\int_A w(z)\,\mathrm dz,
  \qquad
  \mu_w(\mathrm dz)=\frac{w(z)\1_A(z)}{W}\,\mathrm dz.
\]
The weighted $C$-ball floating body $F_C(A,w,\delta)$ is defined as
the intersection of $C+x$ over all $x$ such that
$\int_{A\setminus(C+x)}w(z)\,\mathrm dz\le\delta$.

\begin{proposition}[Weighted floating-body identification]
  \label{prop:weighted-floating-identification}
  Under the assumptions above, for every $\alpha\in(0,1)$ we have
  $D_\alpha^C(\mu_w)=F_C(A,w,\alpha W)$.
\end{proposition}
\begin{pf}
  The high-mass condition is equivalent to
  $\int_{A\setminus(C+x)}w(z)\,\mathrm dz<\alpha W$. The weighted loss
  is continuous in $x$: this follows by dominated convergence, since
  the boundary of every translate of the convex body $C$ has Lebesgue
  measure zero. It remains to replace the strict inequality by a
  nonstrict one. Fix a covering shift $x_0$ with $A\subset C+x_0$, and
  let $x_t=(1-t)x+tx_0$. Convexity gives
  $(1-t)(A\cap(C+x))+tA\subset A\cap(C+x_t)$ and also
  $A\cap(C+x)\subset A\cap(C+x_t)$. If the weighted loss at $x$ lies
  strictly between $0$ and $W$, then $|A\cap(C+x)|<|A|$. The
  Brunn--Minkowski inequality therefore shows that the last inclusion
  has a difference of positive volume for every $t\in(0,1)$. Since $w$
  has a positive minimum on the compact set $A$, the weighted overlap
  strictly increases. Hence every point of a nonstrict loss level is a
  limit of points in the corresponding strict sublevel set. Intersecting translates over a
  shift set or over its closure gives the same closed set. The
  expression for $D_\alpha^C(\mu_w)$ follows from
  Proposition~\ref{prop:erosion}. 
\end{pf}

\subsection{Uniform distribution on the testing body}
\label{sec:unif-distr-test}

Assume now that $\mu$ is uniform on $K$ itself. Its covariogram is $g_K(x)=|K\cap(K+x)|$, and $f_{K,\mu}(x)= g_K(x)/|K|$. The maximum is $1$ and is attained uniquely at $x=0$. Thus
$S_K(\mu)=\Int K$ by Proposition~\ref{prop:effective-support}.  The
function $g_K$ is even, and $g_K^{1/d}$ is concave on its support.
For $0\le\beta<1$, put $U_\beta=\cl\{x:g_K(x)>\beta|K|\}$. Then $U_\beta$ is compact, convex, and centrally symmetric. For
$0<\beta<1$ it equals $\{x:g_K(x)\ge\beta|K|\}$, while
$U_0=K + \check{K}$. Consequently, for $\alpha\in(0,1]$, we get $D_\alpha^K(\mu)=K\ominus U_{1-\alpha}
  =\{y:y+U_{1-\alpha}\subset K\}$. The endpoint $\alpha=0$ is not included: if the strict-level
definition is extended formally to $\alpha=0$, then $D_0^K=\R^d$,
whereas the right limit $D_{0+}:=\cl\bigcup_{\alpha>0}D_\alpha^K$ equals
$K$.

For $K=[0,1]$ in $\R$, $g_K(x)=1-|x|$ for $|x|\le1$, so
$D_\alpha^K(\mu)=[\alpha,1-\alpha]$, for all $0<\alpha\le\frac12$, and
the region is empty for $\alpha>1/2$. The depth is
$\de_K(y;\mu)= \min(y,1-y)$ for $0\le y\le1$, and $\de_K(y;\mu)= 0$
for $y\notin[0,1]$. An example of a bivariate $K$-depth region with
$\mu$ uniform on $K$ is in Figure~\ref{fig:uniform-K}.

\begin{figure}[ht]
\centering
\begin{subfigure}{0.43\textwidth}
 \centering\includegraphics[width=\linewidth]{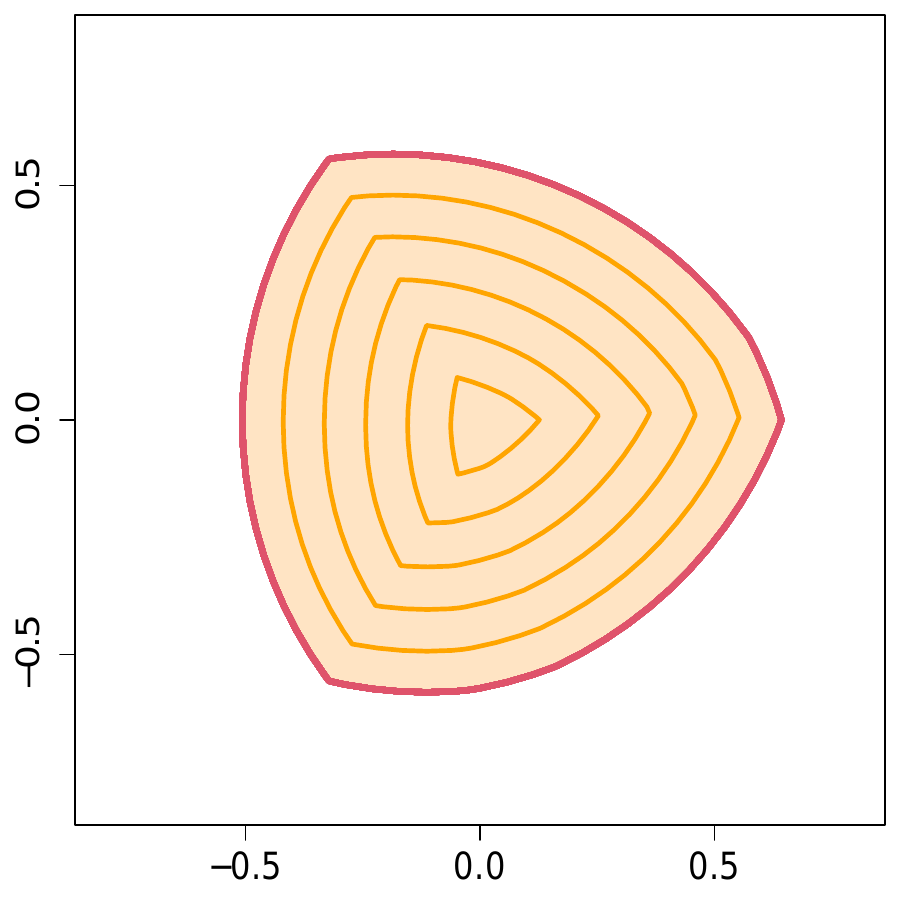}
 \caption{Nested depth regions.}
\end{subfigure}\hfill
\begin{subfigure}{0.43\textwidth}
 \centering\includegraphics[width=\linewidth]{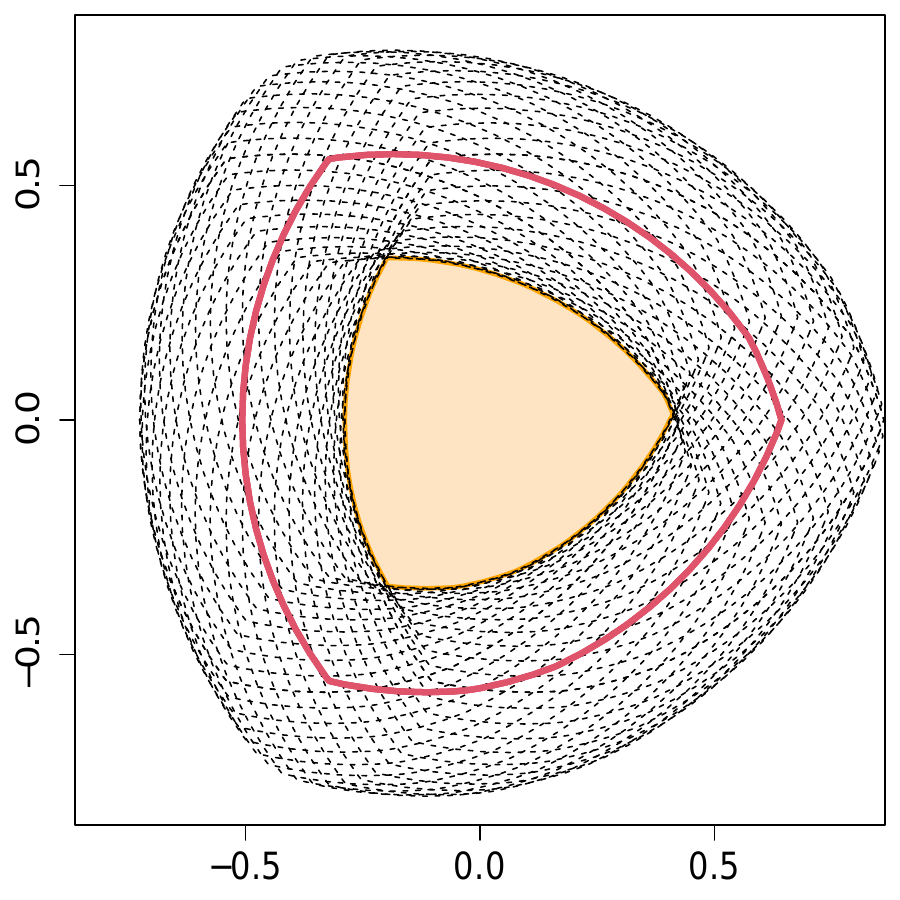}
 \caption{Active translates at one level.}
\end{subfigure}
\caption{Section~\ref{sec:unif-distr-test}: $K$-depth regions when the distribution is uniform on a
nonsymmetric testing body $K$.}
\label{fig:uniform-K}
\end{figure}

\section{Empirical depth and computation}

Let $X_1,X_2,\dots$ be independent observations with common law
$\mu$. Throughout this section the testing body $K$ is fixed and
deterministic unless explicitly stated otherwise.

\subsection{Finite-sample representation}

For a shift $x$, define its coverage count by
$N_n(x):=\sum_{i=1}^n\1_{\{X_i\in K+x\}}$, so that $n^{-1}N_n$ is the
translation-mass function generated by the empirical measure
$\mu_n:=\frac1n\sum_{i=1}^n\delta_{X_i}$
Hence, with $m_n(\alpha)=\lfloor n(1-\alpha)\rfloor+1$, we have
\begin{align}
  L^>_{1-\alpha,n}
   =\{x:N_n(x)\ge m_n(\alpha)\}, \quad \mbox{and} \quad  D_\alpha^K(\mu_n)
  =K\ominus\cl\conv L^>_{1-\alpha,n}.
    \label{eq:empirical-erosion}
\end{align}
Corollary~\ref{cor:empirical-centerpoint} also gives the exact
sample-independent guarantee
$D_{\lceil n/h_T(K)\rceil/n}^K(\mu_n)\ne\varnothing$, and shows that
this level cannot be increased uniformly over all samples.

If the sample is contained in at least one translate of $K$, then, for
every $0<\alpha\le1/n$ we have
$D_\alpha^K(\mu_n)=\operatorname{conv}_K\{X_1,\dots,X_n\}$.  Indeed, for
$0<\alpha\le1/n$, the strict inequality $\mu_n(K+x)>1-\alpha$ holds
exactly when $K+x$ contains all $n$ observations. Intersecting those
translates gives the sample $K$-hull.

\subsection{Uniform consistency}

The deterministic estimate in Theorem~\ref{thm:joint-continuity}
applies directly to empirical measures.

\begin{theorem}
  \label{thm:uniform-consistency}
  If $\mu$ is $K$-smooth, then
  $\sup_{y\in\R^d} \bigl|\de_K(y;\mu_n)-\de_K(y;\mu)\bigr| \to 0$
  almost surely.  More generally, the same conclusion holds whenever
  the translate class $\{K+x:x\in\R^d\}$ is $\mu$-Glivenko--Cantelli.
\end{theorem}
\begin{pf}
  The empirical measures converge weakly to $\mu$ almost surely. Under
  $K$-smoothness, Lemma~\ref{lem:uniform-translates} and
  \eqref{eq:depth-lipschitz-measure} give the result. Under the
  alternative assumption, the right-hand side of
  \eqref{eq:depth-lipschitz-measure} converges to zero by the
  definition of a Glivenko--Cantelli class.  The uniform
  continuity-set principle used here goes back to \cite{RangaRao1962}.
\end{pf}

Note that the family of translates of any convex body is a
Glivenko--Cantelli class in dimension $d=2$, see \cite{NaszodiTaschuk2010}.

Uniform convergence of the functions does not by itself guarantee
Hausdorff convergence at a plateau level. The correct condition is
continuity of the nested population regions at the chosen level.

\begin{corollary}
  Maintain the assumptions of Theorem~\ref{thm:uniform-consistency},
  and assume additionally that $D_\alpha^K(\mu)$ is a nonempty compact
  set and
  \begin{equation}
    \lim_{\eta\downarrow0}
    d_H(D_{\alpha+\eta}^K(\mu),D_\alpha^K(\mu))=0,
    \qquad
    \lim_{\eta\downarrow0}
    d_H(D_{\alpha-\eta}^K(\mu),D_\alpha^K(\mu))=0.
    \label{eq:H-regular}
  \end{equation}
  Then $d_H(D_\alpha^K(\mu_n),D_\alpha^K(\mu))\to 0$ almost surely.
\end{corollary}
\begin{pf}
  Let $\Delta_n=\sup_y|\de_K(y;\mu_n)-\de_K(y;\mu)|$. For every
  $\eta>\Delta_n$,
  \begin{equation}
    D_{\alpha+\eta}^K(\mu)
    \subset D_\alpha^K(\mu_n)
    \subset D_{\alpha-\eta}^K(\mu).
    \label{eq:region-sandwich}
  \end{equation}
  Choose a deterministic sequence $\eta_n\downarrow0$ with
  $\Delta_n<\eta_n$ eventually. The two Hausdorff limits in
  \eqref{eq:H-regular}, together with the sandwich
  \eqref{eq:region-sandwich}, imply the assertion.
\end{pf}

\subsection{Pointwise weak limits}

We next describe the effect of nonunique active translates. Fix
$y\in\R^d$, and let
\begin{align*}
  O_y:=\{x:y\notin K+x\}
  =\R^d\setminus(y+\check K), \quad \mbox{and} \quad F_y :=\cl O_y
  =\R^d\setminus\Int(y+\check K).
\end{align*}
If $\mu$ is $K$-smooth, continuity of $f=f_{K,\mu}$ gives
$M_K(y;\mu)=\max_{x\in F_y}f(x)$. Define the compact active set
$\mathcal A_y =\argmax_{x\in F_y}f(x)$. For a compact parameter set
$T$, write
\begin{displaymath}
  \Gc_n(x)=\sqrt n\{\mu_n(K+x)-\mu(K+x)\}, \quad x\in T.
\end{displaymath}
The translate class restricted to $T$ is called $\mu$-Donsker if
$\Gc_n$ converges in $\ell^\infty(T)$ to the centred Gaussian process
$\Gc$ with covariance
$\operatorname{Cov}(\Gc(x),\Gc(x')) =\mu((K+x)\cap(K+x'))-f(x)f(x')$.

\begin{theorem}[Pointwise empirical limit]
  Assume that $\mu$ is $K$-smooth and that $M_K(y;\mu)>0$. Choose a
  compact $T$ containing $\mathcal A_y$ in its interior and satisfying
  $\sup_{x\in F_y\setminus T}f(x)<M_K(y;\mu)$. Suppose that the
  translate class restricted to $T$ is $\mu$-Donsker.  Then
  $\sqrt n\{\de_K(y;\mu_n)-\de_K(y;\mu)\}$ converges in distribution
  to $-\sup_{x\in\mathcal A_y}\Gc(x)$.
\end{theorem}
\begin{pf}
  The localisation gap and uniform consistency of the translate masses
  imply that, with probability tending to one, the empirical supremum
  defining $M_K(y;\mu_n)$ may be restricted to $O_y\cap T$.
  On $\ell^\infty(T)$ consider for $h \in \ell^\infty(T)$ the functional $\Phi_y(h)=\sup_{x\in O_y\cap T}h(x)$. Because $O_y$ is dense in $F_y$, $f$ is continuous, and its
  maximisers on $F_y\cap T$ form $\mathcal A_y$, the supremum
  functional is Hadamard directionally differentiable at $f$,
  tangentially to the space of continuous functions $C(T)$, with derivative
  \begin{equation}
    \Phi'_{y,f}(g)=\sup_{x\in\mathcal A_y}g(x).
    \label{eq:sup-derivative}
  \end{equation}
  For completeness, the upper bound follows by taking cluster points
  of near-maximisers. For the lower bound, approximate any
  $x\in\mathcal A_y$ by points of $O_y$ so closely that the loss in
  $f$ is $o(t)$ along a perturbation of size $t$.

  The $K$-smoothness assumption makes $x\mapsto\1_{K+x}$ continuous in
  $L_2(\mu)$. Hence the Donsker limit has a version whose sample paths
  are continuous on $T$ in the induced semimetric, so it belongs to
  the tangent space in \eqref{eq:sup-derivative}. The directional
  delta method \cite{FangSantos2019} shows that
  $\sqrt n\{M_K(y;\mu_n)-M_K(y;\mu)\}$ converges in distribution to
  $\sup_{x\in\mathcal A_y}\Gc(x)$.  Since the depth is one minus this
  supremum, the result follows.
\end{pf}

\begin{remark}
  The Donsker assumption is a genuine empirical-process condition and
  is not implied merely by convexity of $K$. In dimension $d=2$, the
  family of positive homothets of any convex body has VC dimension at
  most $3$ \cite[Theorem~3]{NaszodiTaschuk2010}; in particular, the
  translate class is VC and, under the standard measurability conditions,
  is Donsker by the VC entropy bound and entropy criterion from
  \cite[Theorems~2.6.7 and~2.5.2]{VanderVaartWellner1996}. In contrast,
  \cite[Example~4]{NaszodiTaschuk2010} constructs a convex body in
  $\R^3$ whose family of translates has infinite VC dimension. Thus
  convexity alone does not yield a universal Donsker theorem in
  dimensions $d\ge3$.
\end{remark}

\subsection{Computation, examples and applications}
\label{sec:plan-comp-exampl}

When $K$ is a polytope,
formulas~\eqref{eq:empirical-erosion} lead to a direct arrangement
algorithm.
\begin{enumerate}[leftmargin=2.2em]
\item Construct the arrangement of the reflected translates
  $X_i+\check K$ in shift space.
\item Label each arrangement cell by its constant coverage count
  $N_n(x)$.
\item Take the union of cells with count at least $m_n(\alpha)$ and
  compute the closed convex hull of that union.
\item Erode $K$ by the resulting convex set, using
  \eqref{eq:empirical-erosion}.
\end{enumerate}
This procedure is exact in polygonal arithmetic in the plane; a numerical
implementation must additionally report its tolerance for incidences,
convex-hull construction, and erosion.

\begin{figure}[ht]
\centering
\begin{subfigure}{0.425\textwidth}
 \centering\includegraphics[width=\linewidth]{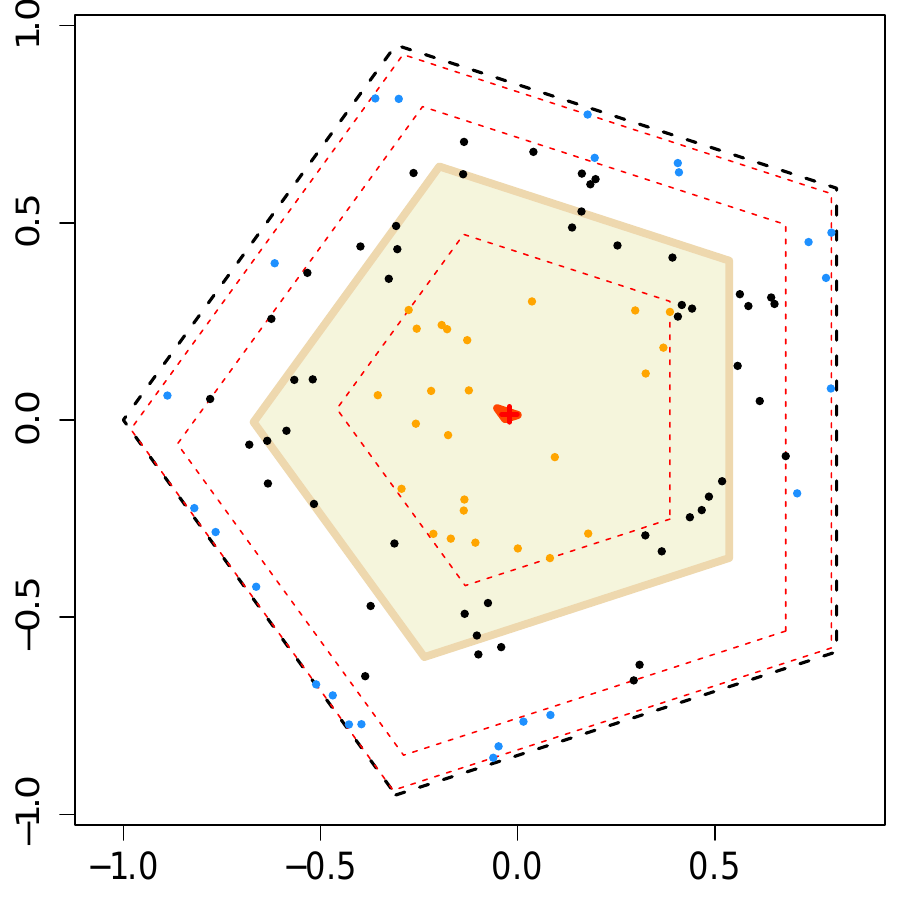}
 \caption{$n=100$.}
\end{subfigure}\hfill
\begin{subfigure}{0.425\textwidth}
 \centering\includegraphics[width=\linewidth]{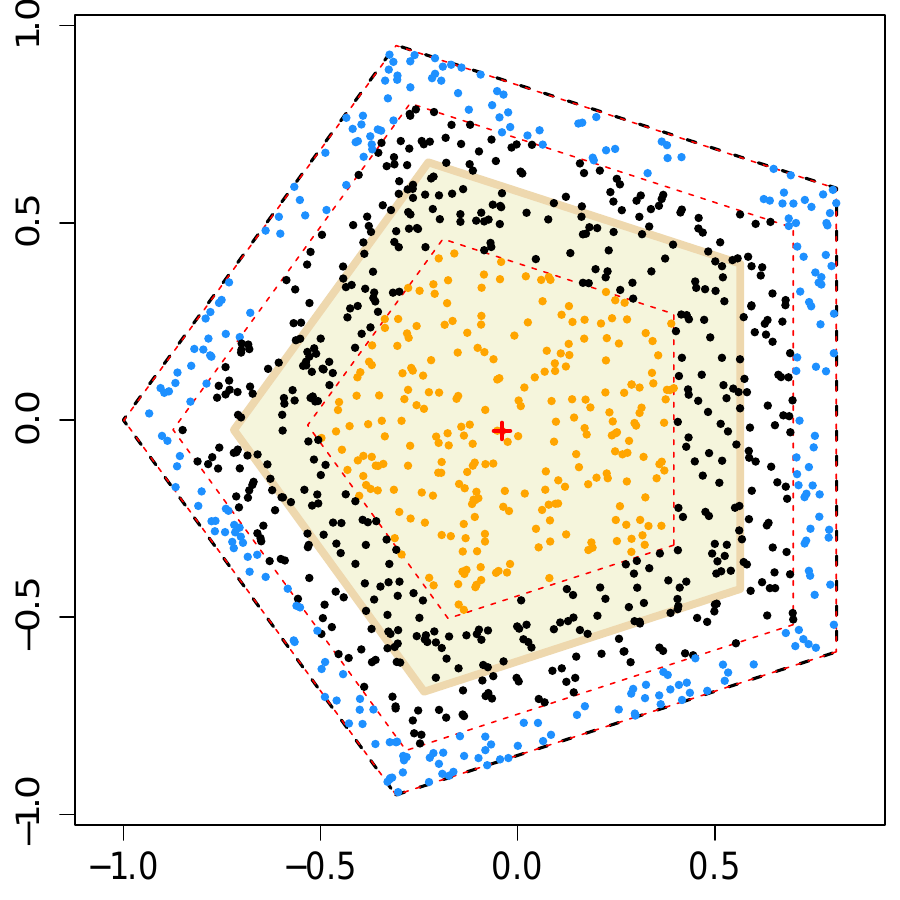}
 \caption{$n=1000$.}
\end{subfigure}

\begin{subfigure}{0.425\textwidth}
 \centering\includegraphics[width=\linewidth]{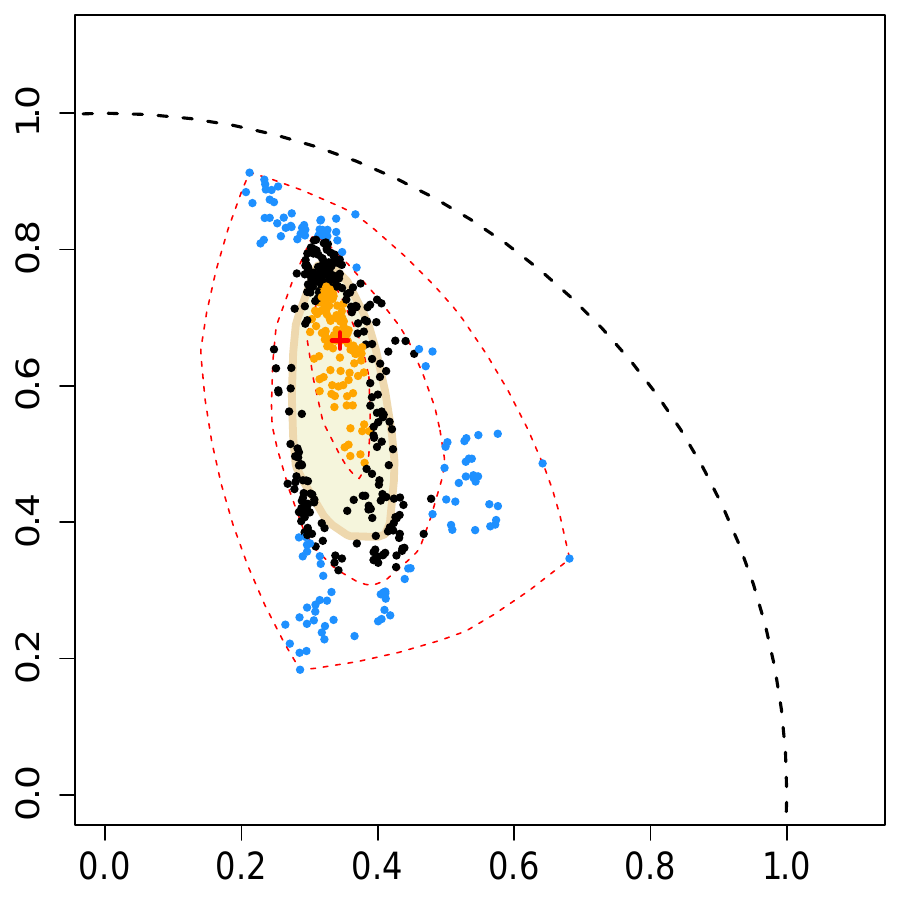}
 \caption{Hydrochemical (compositional) data in Hellinger coordinates supported in $B_1$.}
\end{subfigure}\hfill
\begin{subfigure}{0.425\textwidth}
 \centering\includegraphics[width=\linewidth]{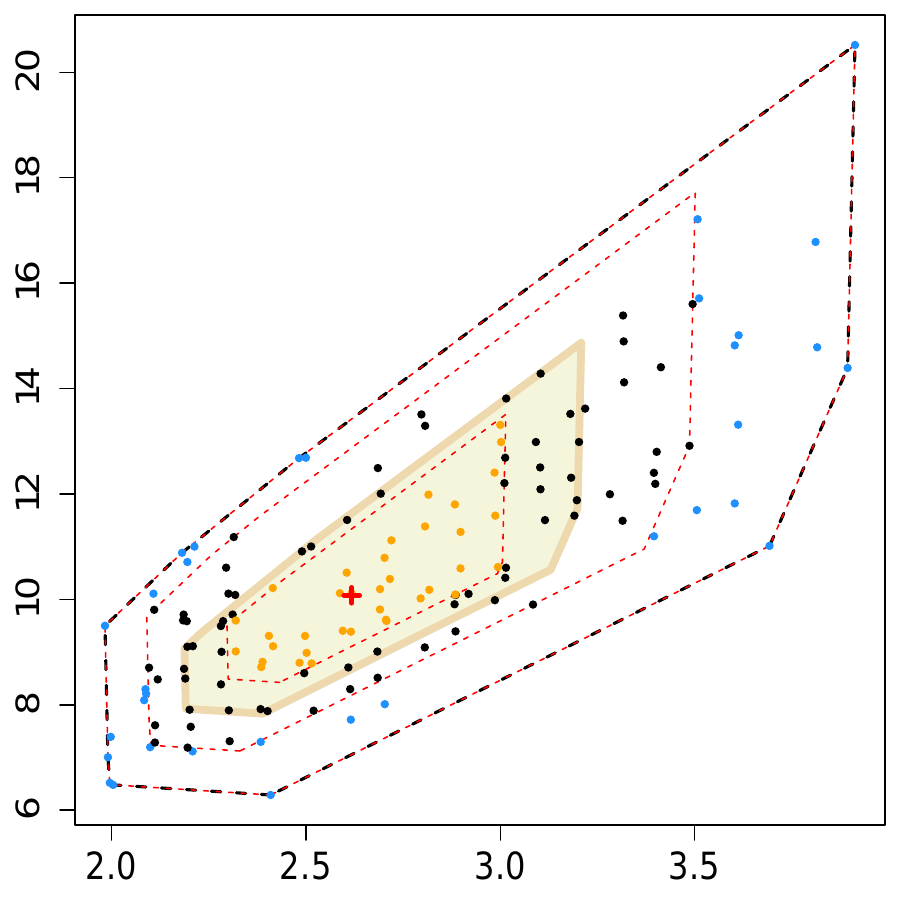}
 \caption{Cat weights dataset with data-dependent $K$ (the convex hull of the data).}
\end{subfigure}
\caption{Section~\ref{sec:plan-comp-exampl}: Top: Empirical $K$-depth regions for uniform samples in a regular pentagon. The same deterministic testing body is used in both panels (a) and (b). Bottom: Empirical $K$-depth regions for two data sets. Panel (c) uses a deterministic support-adapted body $B_1$; panel (d) is an exploratory sample-dependent construction.}
\label{fig:sample-data}
\end{figure}

Figure~\ref{fig:sample-data} illustrates four choices of support or
testing body. In each panel the (outermost) black dashed curve is $K$,
thin dashed curves are selected empirical regions, the beige region
with the thick boundary is the smallest $K$-depth region containing
$50\%$ of the data, and the red cross (the $K$-median) is the barycentre of the
smallest nonempty region. The beige region represents the $K$-adapted central part of the data cloud and serves as a boxplot-like visualisation of multivariate datasets \cite{Rousseeuw_etal1999}. The $K$-median presents a robust location estimator respecting the geometry of the testing body; the bound from Theorem~\ref{thm:sharp-centerpoint} translates into a quantitative robustness result \cite{DonohoGasko1992}. Our results from Section~\ref{sec: dimension} guarantee non-degeneracy of the $K$-depth even in high-dimensional data setups, and the quantitative halfspace approximation from Section~\ref{subsec:quantitative-dilation} offers well-behaved $K$-depth variants of the halfspace depth in such situations. The $K$-depth with $K$ a (hyper)rectangle or a ball yields a principled alternative to the box/band depth \cite{Gil_etal1992, Lopez_Romo2009} and the spherical depth-like constructions \cite{Elmore_etal2006, Liu_Modarres2011} used in contemporary statistics and machine learning \cite{Kleindressner_Luxburg2017}.

The last panel of
Figure~\ref{fig:sample-data} uses the sample convex hull as $K$. It is
included as an exploratory data display; because that testing body is
random, the fixed-$K$ consistency
Theorem~\ref{thm:uniform-consistency} does not apply to it without an
additional stability analysis.

\section{Discussion}

The construction studied here has two complementary interpretations.
In shift space, it retains translations of $K$ with large probability
content. In observation space, it intersects the retained translates.
The identity $D_\alpha^K(\mu)=K\ominus\cl\conv L^>_{1-\alpha}$
connects these viewpoints and is the main geometric mechanism behind
the theory.

Several distinctions are essential. The $K$-depth has the baseline
$1-\beta_{K,\mu}$ and should be normalised when a depth vanishing off
its effective support is desired. The closed maximiser core and the
effective support coincide only under additional conditions and are
not interchangeable. Smoothness of $K$ alone does not imply smoothness
of a depth region: uniqueness and stability of active translations are
the decisive local conditions. Finally, a smooth testing body recovers
halfspace depth only in the large-dilation limit; nonsmooth bodies retain
their tangent-cone geometry.

The translative Helly number $h_T(K)$ gives an exact global measure of this
dependence on $K$: the worst-case maximal depth is $1/h_T(K)$, and the
same mechanism yields the exact finite-sample constant. At the local
large-scale end, the rolling-ball estimate controls the approximation
to halfspace depth by the mass of a slab of width of order $t^{-1}$.
If the density is bounded,
Theorem~\ref{thm:quantitative-dilation} gives the explicit bound 
whose $t^{-1}$ order is already sharp for concentric Euclidean
balls. Thus the geometry of the testing body affects both the global
centerpoint guarantee and the quantitative bias relative to halfspace
depth.

A deterministic testing body should normally encode known support or
shape constraints. If affine equivariance is required, $K$, the data,
and the query point must be transformed jointly as in
Proposition~\ref{prop:equivariance}; estimating $K$ from the same sample
makes the testing body random.

The choice of $K$ therefore controls a genuine modelling tradeoff. A
large smooth body makes the depth close to halfspace depth, with the
rate described in Theorem~\ref{thm:quantitative-dilation}. A
polyhedral or nonsmooth body can encode directional constraints and
may have a larger worst-case centerpoint constant, as in the
parallelotope case.  This dependence is not reducible to a simple
inclusion principle, while Theorem~\ref{thm:K-stability} gives the
basic continuity needed when $K$ is chosen or estimated approximately.

The identification with the recent ball-floating-body programme is
complementary rather than merely terminological. The results of
Sch\"utt, Werner and Yalikun \cite{SchuettWernerYalikun2025} provide
sharp geometric volume derivatives, and curvature functionals for
smooth uniform or weighted models. The present framework supplies the
probabilistic extension to arbitrary measures, continuity under weak
convergence, empirical process limits, and the transition to halfspace
depth under dilation. The new ball-illumination theory of Oliva,
Werner and Yalikun \cite{OlivaWernerYalikun2026} provides the
corresponding exterior branch; it shares the relative curvature
functionals but is not representable by the inward intersection that
defines a depth region.

The results for empirical measures are formulated for a fixed
testing body. The stability theorem above supplies the deterministic
part needed for data-driven choices of $K$ when an independent argument
shows $d_H(K_n,K)\to0$. Uniform inference for entire depth contours
and differentiating depth regions under perturbations of $K$ would require
additional nondegeneracy assumptions. These problems are not treated here
and are left for future work.


\bibliographystyle{cas-model2-names}

\bibliography{allpapers_v2}



\end{document}